\documentclass[11pt,reqno]{amsart} 
\usepackage{amsmath,tikz,multicol,enumitem}
\usepackage[utf8]{inputenc}
\usepackage{hyperref,longtable,multirow}
\usepackage[normalem]{ulem}
\usepackage{amsmath}
\usepackage{amssymb}
\usepackage{amsthm}
\usepackage{mathtools}
\usepackage{mathrsfs}
\usepackage{verbatim}
\usepackage{xcolor}
\usepackage{fullpage}
\usepackage[skip=3pt, indent=20pt]{parskip}
\usepackage{caption}
\usepackage{textcomp}
\usepackage{multirow}
\usepackage{MnSymbol}
\usepackage{ytableau}

\DeclareMathOperator*{\argmax}{arg\,max}
\DeclareMathOperator*{\argmin}{arg\,min}

\newcommand{\Sym}{\ensuremath{\mathfrak{S}}}
\newcommand{\Des}{\mathrm{Des}}
\newcommand{\Asc}{\mathrm{Asc}}

\newcommand{\flick}{\mathrm{flick}}

\newcommand{\A}{\texttt{A}}
\newcommand{\D}{\texttt{D}}
\newcommand{\T}{\texttt{T}}

\newcommand{\NDPF}{\mathrm{PF}^{\uparrow}}
\newcommand{\PF}{\mathrm{PF}}
\newcommand{\x}{\mathbf{x}}
\newcommand{\y}{\mathbf{y}}

\newcommand{\des}{\mathrm{des}}
\newcommand{\NN}{\mathbb{N}}

\theoremstyle{definition}

\newtheorem{theorem}{Theorem}[section]
\newtheorem{proposition}[theorem]{Proposition}
\newtheorem{lemma}[theorem]{Lemma}
\newtheorem{problem}[]{Problem}
\newtheorem{corollary}[theorem]{Corollary}
\newtheorem{remark}[]{Remark}

\newtheorem{definition}[theorem]{Definition}
\newtheorem{example}[theorem]{Example}

\newcommand{\Tie}{\mathrm{Tie}}
\newcommand{\Val}{\mathrm{Val}}
\newcommand{\PTPF}{\mathrm{PTPF}}
\newcommand{\VTPF}
{\ensuremath{\mathrm{VTPF}}}

\title{On some discrete statistics of parking functions}

\author[Cruz]{Ari Cruz}
\author[Harris]{Pamela E. Harris}\thanks{P.~E.~Harris was supported in part by a Karen Uhlenbeck EDGE Fellowship.}
\author[Harry]{Kimberly J. Harry}
\author[Kretschmann]{Jan Kretschmann}
\author[McClinton]{Matt ~McClinton}
\author[Moon]{Alex Moon}
\author[Museus]{John O. Museus}
\author[Redmon]{Eric Redmon}

\address[A.~Cruz, P.~E.~Harris, K.~J.~Harry, J.~Kretschmann, M.~McClinton, A.~Moon, J.~O. Museus]{Department of Mathematical Sciences, University of Wisconsin-Milwaukee, WI 53211}
\email{\textcolor{blue}{\href{mailto:aristeo@uwm.edu}{aristeo@uwm.edu}}, \textcolor{blue}{\href{mailto:peharris@uwm.edu}{peharris@uwm.edu}}, \textcolor{blue}{\href{mailto:kjharry@uwm.edu}{kjharry@uwm.edu}},
\textcolor{blue}{\href{mailto:kretsc23@uwm.edu}{kretsc23@uwm.edu}}, \textcolor{blue}{\href{mailto:mcclin33@uwm.edu}{mcclin33@uwm.edu}},
\textcolor{blue}{\href{mailto:ajmoon@uwm.edu}{ajmoon@uwm.edu}}, \textcolor{blue}{\href{mailto:jomuseus@uwm.edu}{jomuseus@uwm.edu}}}
\address[E.~Redmon]{Department of Mathematics and Statistics, Marquette University, Milwaukee, WI 53233}
\email{\textcolor{blue}{
\href{mailto:eric.redmon@marquette.edu}{eric.redmon@marquette.edu}}}

\begin{document}
\begin{abstract}
A parking function is a sequence $\alpha=(a_1,a_2,\ldots,a_n)\in[n]^n$ where its nondecreasing rearrangement
$\beta=(b_1,b_2,\ldots,b_n)$ satisfies $b_i\leq i$ for all $1\leq i\leq n$. In this article, we study parking functions based on their ascents (indices at which $a_i<a_{i+1}$),  descents (indices at which $a_i>a_{i+1}$), and ties (indices at which $a_i=a_{i+1}$).
By using multiset Eulerian polynomials, we give a generating function for the number of parking functions of length $n$ with $i$ descents. 
We present a recursive formula for the number of parking functions of length $n$ with descents at a specified subset of $[n-1]$. 
We establish the set of parking functions with descent set $I$ and the set of parking functions with descent set $J=\{n-i : i\in I\}$ are in bijection, and hence these sets have the same cardinality. 
As a special case, we show that the number of parking functions of length $n$ with descents at the first $k$ indices is given by $\frac{1}{n}\binom{n}{k}\binom{2n-k}{n-k-1}$. We prove this by bijecting to the set of standard Young tableaux of shape $((n-k)^2,1^k)$, which are enumerated by the same formula. 
We also study peaks and valleys of parking functions, which are indices at which $a_{i-1}<a_i>a_{i+1}$ and $a_{i-1}>a_i<a_{i+1}$, respectively. We show that the set of parking functions with no peaks and no ties is enumerated by the Catalan numbers, and the set of parking functions with no valleys and no ties is enumerated by the Fine numbers. 
We conclude our study by characterizing when a parking function is uniquely determined by its statistic encoding; a word indicating what indices in the parking function are ascents, descents, and ties. 
We provide open problems throughout.

\end{abstract}

\maketitle

\section{Introduction}
Throughout, we let $\NN=\{1,2,3,\ldots\}$ and $[n]=\{1,2,\ldots,n\}$ for $n\in\NN$, and we define $[0]=\emptyset$.
Moreover, we let $\mathfrak{S}_n$ denote the set of permutations of $[n]$ and use one-line notation $\pi=\pi_1\pi_2\cdots\pi_n$ to denote elements of~$\mathfrak{S}_n$.
If $\pi_{i}>\pi_{i+1}$, then  $\pi$ has a \textit{descent} at index $i$; and if $\pi_{i}<\pi_{i+1}$, then $\pi$ has an \textit{ascent} at index~$i$.
There is a long history of studying and enumerating permutations with certain descents and ascents.
This dates back to the work of MacMahon, who established that for a fixed set $I$, with $n \geq \max(I)$ varying, the number of permutations in $S_n$ with descent set $I$ is a polynomial in $n$~\cite[Art. 157]{MacMahon}. 
These polynomials are often referred to as \textit{descent polynomials} and studying the coefficients of these polynomials and their roots has received attention in the literature \cite{Bencs,DescentPolys}. Moreover, there is ample work on generalizing these findings to multipermutations \cite{Elizalde,jiradilok2021roots,Raychev}. 
In the present, we are motivated by the work of 
Schumacher who, enumerated descents, ascents, and ties in parking functions \cite{Schumacher}. 

\begin{figure}[htbp]
    \centering
\begin{tikzpicture}[yscale=.5, xscale=.75]
\draw[thick] (0.5,6.5)--(0.5,0.5)--(9.5,0.5);
\node at (0,1)[scale=.75]{$1$};
\node at (0,2)[scale=.75]{$2$};
\node at (0,3)[scale=.75]{$3$};
\node at (0,4)[scale=.75]{$4$};
\node at (0,5)[scale=.75]{$5$};
\node at (0,6)[scale=.75]{$6$};
\draw[thick] (1,3)--(2,1)--(3,5)--(4,6)--(5,4)--(6,3)--(7,3)--(8,2)--(9,1);
\node at (1,3){$\bullet$};
\node at (2,1){$\bullet$};
\node at (3,5){$\bullet$};
\node at (4,6){$\bullet$};
\node at (5,4){$\bullet$};
\node at (6,3){$\bullet$};
\node at (7,3){$\bullet$};
\node at (8,2){$\bullet$};
\node at (9,1){$\bullet$};
\node at (1,0)[scale=.75]{$1$};
\node at (2,0)[scale=.75]{$2$};
\node at (3,0)[scale=.75]{$3$};
\node at (4,0)[scale=.75]{$4$};
\node at (5,0)[scale=.75]{$5$};
\node at (6,0)[scale=.75]{$6$};
\node at (7,0)[scale=.75]{$7$};
\node at (8,0)[scale=.75]{$8$};
\node at (9,0)[scale=.75]{$9$};
\end{tikzpicture} 
\caption{Graph of the parking function $(3,1,5,6,4,3,3,2,1)\in\PF_{9}$.}
    \label{fig:enter-label}
\end{figure}

A parking function is a sequence $\alpha=(a_1,a_2,\ldots,a_n)\in[n]^n$ where its nondecreasing rearrangement
$\beta=(b_1,b_2,\ldots,b_n)$ satisfies $b_i\leq i$ for all $1\leq i\leq n$.
Furthermore, recall that parking functions encode the parking preferences for $n$ cars in a queue attempting to park on a one-way street with $n$ parking spots numbered sequentially. For $i\in [n]$, car $i$ enters the street and attempts to park in spot $a_i$. 
If that parking spot is available it parks.
If the parking spot is occupied, then the car seeks forward attempting to park on the first available spot it encounters.
If no such a spot exists, we say parking fails. 
If all cars are able to park under this parking scheme, then $\alpha$ is called a parking function.

We let $\PF_n$ denote the set of parking functions of length $n$. 
If $\alpha=(a_1,a_2,\ldots,a_n)\in\PF_n$, then $\alpha$ has descents and ascents defined in the analogous way as for  permutations, and $\alpha$ has a \textit{tie} at $i$ if $a_i=a_{i+1}$. 
In Figure \ref{fig:enter-label}, we plot $(i,a_i)$ to illustrate 
that the parking function $\alpha=(3,1,5,6,4,3,3,2,1)\in\PF_{9}$
has descents at 1, 4, 5, 7, and 8, ascents at 2 and 3, and a tie at 6.
Let $\des(\alpha)$ denote the number of descents in $\alpha$.
Schumacher shows that
the number of descents among all parking functions of length $n$ is  $\binom{n}{2}(n+1)^{n-2}$, see \cite[Theorem 10]{Schumacher}.
If $\PF_{(n,i)}$ denotes the set of parking functions with exactly $i$ ties, then \cite[Theorem 13 and Lemma 17]{Schumacher} state
\begin{align}\label{rowsums}
|\PF_{(n,i)}|&=\binom{n-1}{i}n^{n-1-i}\\
\intertext{and}
\sum_{\alpha\in\PF_{(n,i)}}\des(\alpha)&=\frac{n-1-i}{2}\binom{n-1}{i}n^{n-1-i},\label{thm:pfs with i ties}
\end{align}
respectively.
Let $T_n(i,j)$ denote the number of parking functions of length $n$ with $i$ ties and $j$ descents. If $i+j\geq n$, then $T_n(i,j)=0$. Table \ref{recreate table} provides the values of $T_6(i,j)$ arranged in a triangular array such that 
$i$ decreases from top to bottom and $j$ decreases from left to right.\footnote{Table \ref{recreate table} first appeared \cite[Figure 2]{Schumacher}, however there was a typographical error for the value $T_6(3,1)$, which should be 260.} For further data in this triangle see \cite[\href{https://oeis.org/A333829}{A333829}]{OEIS}.

\begin{figure}[htbp]
    \begin{tabular}{ccccccccccc}
         &&&&&1&&&&&\\
         &&&&15&&15&&&&\\
         &&&50&&260&&50&&&\\
         &&50&&1030&&1030&&50&&\\
         &15&&1240&&3970&&1240&&15&\\
         1&&407&&3480&&3480&&407&&1\\
    \end{tabular}
    \caption{Parking function distribution for $n = 6$.}\label{recreate table}
\end{figure}
For general $n$, the numbers along the diagonal edges of the triangular array in Figure \ref{recreate table} are the Narayana numbers \cite[Theorem 12]{Schumacher}
\[T_n(i,0)=T_n(n-1-i,i)=\frac{1}{i+1}\binom{n}{i}\binom{n-1}{i},\]
and the row sums are given by \eqref{rowsums}.
General closed formulas for the values of $T_n(i,j)$ remain {unknown}.
However, in Section \ref{sec:eulerian}, we utilize multiset Eulerian polynomials to give a generating function for the values $\sum_{i=0}^nT_n(i,j)$. Our first result follows.

\begin{itemize}[leftmargin=.15in]

    \item Theorem \ref{thm:eulerian}:
    Consider the generating function
$    \sum_{n =1}^{\infty}\sum_{j=1}^{n-1}d(n,j)y^jx^n,
$
where $d(n,j)$ is the number of parking functions in $\PF_n$ that have $j$ descents and let $\NDPF_n$ denote the set of nondecreasing parking functions of length $n$. Then \begin{align}
\sum_{n=1}^{\infty}\sum_{j=1}^{n-1} d(n,j)x^ny^j = \sum_{n=1}^{\infty} \left(\sum_{\beta \in \NDPF_n} A_{\beta}(y) \right) x^n,
\end{align}
where $A_\beta(y)$ denotes the multiset Eulerian polynomial in $y$ on all multipermutations of $\beta$. 
\end{itemize}
We prove Theorem \ref{thm:eulerian} in Section \ref{sec:eulerian}.

Now we will turn our attention to \textit{descent sets} of parking functions. 
Given a multiset $X$ of size $n$ with elements in $[n]$, we  let $W_n(X)$ denote the set of words of length $n$ that are constructed using the letters in $X$. 
In general, we denote a word of length $n$ as $w=w_1w_2\cdots  w_n$.
If $I\subseteq[n-1]$, then the set of words in $W_n(X)$ whose descent set is exactly $I$ is denoted  
\begin{align*}
D_X(I;n)&=\{w\in W_n(X)\,:\,\Des(w)=I \},
\end{align*}
where $\Des(w)=\{i\in[n-1]: w_{i}>w_{i+1}\}$.
Let $d_X(I;n)=|D_X(I;n)|$.
Let $X$ be the set of nondecreasing parking functions where each $\beta\in\NDPF_n$ is considered as a multiset, denoted as $M(\beta)$.
If $I\subseteq[n-1]$, then
\[D(I;n)=\{\alpha\in\PF_n:\Des(\alpha)=I\}=\bigcupdot_{\beta\in\NDPF_n}D_{M(\beta)}(I;n).\]
We let
$d(I;n)=|D(I;n)|=\sum_{\beta\in\NDPF_n}d_{M(\beta)}(I;n)$.
Note that the set $D(\emptyset;n)=\NDPF_n$,  and it is known that $|\NDPF_n|=C_n=\frac{1}{n+1}\binom{2n}{n}$, the $n$th Catalan number \cite[\href{https://oeis.org/A000108}{A000108}]{OEIS}. 
Hence $d(\emptyset;n)=C_n$.

In Section~\ref{sec:descent sets}, we consider nonempty descent sets $I\subseteq[n-1]$ and establish the following results:
\begin{itemize}[leftmargin=.15in]
    \item Theorem~\ref{thm:I and J}: The set of parking functions with descent set $I$ and the set of parking functions with descent set $J=\{n-i : i\in I\}$ are in bijection, and hence these sets have the same cardinality. Namely, $d(I;n)=d(J;n)$.
    \item Lemma \ref{lem:count selfdual}:
    The number of sets $I \subseteq [n-1]$ such that $I = \{n-i:i \in I\}$ is $2^{\lfloor n/2 \rfloor}$.
    \item Theorem~\ref{thm:recursion}:
    Let $I\subseteq[n-1]$ be nonempty, $m=\max(I)$, and $I^-=I\setminus\{m\}$.  
Then
\begin{align}
d(I;n)=\sum_{\beta\in \NDPF_n}\; \left(\sum_{X\in\mathcal{M}(\beta,m)}d_{X}(I^-;m)\right)-d(I^-;n),\label{eq:recursion}
\end{align}
where, for $\beta\in\NDPF_n$,  $\mathcal{M}(\beta,m)$ denotes  the collection of multisets consisting of $m$ elements of $\beta$.

We remark that Equation \eqref{eq:recursion} is a generalization of \cite[Proposition 2.1]{DescentPolys}, which gives a recursion for the number of permutations with a given descent set. 
\item Proposition~\ref{prop:biject to SYT}: Let $n\geq 1$ and $0 \leq k \leq n-1$. 
If $I=[k]\subseteq[n-1]$, then 
\begin{align}\label{descent binome}
    d(I;n) = \frac{1}{n}\binom{n}{k}\binom{2n-k}{n-k-1}.
\end{align}
The proof of Proposition~\ref{prop:biject to SYT} is given by a bijection between parking functions of length $n$ with descent set $[k]$ and the set of standard Young tableaux of shape $\lambda=((n-k)^2, 1^k)$. 

We recall that Stanley \cite{Stanley:dissections} established a bijection between the set of standard Young tableaux of shape $((d+1)^2,1^{t-1-d})$ and the set of diagonal dissections of a convex $(t+2)$-gon into $d$ regions\footnote{A diagonal dissection is a way to draw $d$ diagonals
in a convex $(n+2)$-gon, such that no two diagonals intersect in their interior}, which are known to be enumerated by $f(t,d)=\frac{1}{t+d+2}\binom{t+d+2}{d+1}\binom{t-1}{d}$ \cite[\href{https://oeis.org/A033282}{A033282}]{OEIS}. 
Hence, Equation~\eqref{descent binome} can be seen as the reindexing $d(I;n) = f(n, n-k-1)$.
\end{itemize}

Another permutation statistic of interest is peaks and valleys. 
Recall that  
when $\pi=\pi_1\pi_2\cdots\pi_n\in\mathfrak{S}_n$, if $\pi_{i-1}<\pi_i>\pi_{i+1}$, then $i$ is a \textit{peak} of $\pi$. 
Similarly if $\pi_{i-1}>\pi_i<\pi_{i+1}$, then $i$ is a \textit{valley} of~$\pi$.
Billey, Burdzy, and Sagan \cite{billey:peak-set} showed that the number of elements in $\mathfrak{S}_n$ with peaks exactly at $I\subset[n-1]$ is given by $p(n)2^{|I|-1}$, where $p(n)$ is a polynomial in $n$ whose degree is one less than the maximum of the set $I$.
The polynomial $p(n)$ is referred to as the \textit{peak polynomial}, and it was conjectured that $p(n)$ has a positive integer expansion in a binomial basis centered at the maximum of $I$. 
This was proven in the affirmative by Diaz-Lopez, Harris, Insko, and Omar in \cite{peakpoly}.
Motivated by these results, we study the set of peakless-tieless parking functions of length $n$, denoted $\PTPF_n$, which is the subset of parking functions which have no peaks and no ties.

In Section \ref{sec:PTPFN}, we prove the following:
\begin{itemize}[leftmargin=.15in]
\item Corollary \ref{cor:peakless tieless catalan}: If $n\geq 1$, then $|\PTPF_{n}|=C_n$, the $n$th Catalan number.    
\item Corollary~\ref{cor:cat triangle}:
    If $\PTPF_{n}(i)=\{\alpha=(a_1,a_2,\ldots,a_n)\in\PTPF_{n}: a_n=i\}$, then 
    \[|\PTPF_{n}(i)|=C_{n,i}\;,
    \]
    where $C_{n,i}$ denotes the $i$th entry in the $n$th row of the Catalan triangle  \cite[\href{https://oeis.org/A009766}{A009766}]{OEIS}. 
\end{itemize} 
We also prove that the number of parking functions of length $n$ with no valleys 
and no ties is the Fine numbers (Theorem \ref{thm:main2}) which correspond to the OEIS entry \cite[\href{https://oeis.org/A000957}{A000957}]{OEIS}.
It remains an open problem to enumerate all parking functions with a particular peak set. 

In Section \ref{sec:statistic encoding}, we study the characterization of parking functions by the locations of their ascents, descents, and ties. 
Calling the statistic encoding of a parking function, we establish the existence of parking functions for every statistic encoding and fully characterize the parking functions which are uniquely identified by their statistic encodings.

\begin{remark}
    We present open problems, labeled ``Problem $n$'', throughout the paper.
\end{remark}

\section{Permutations of Multisets}\label{sec:eulerian}

Recall that a multiset $M$ is denoted as $M=\left\{1^{m_1},\ldots,n^{m_n}\right\}$ where $m_i$ denotes the multiplicity of $i$ for all $i\in[n]$. 
A \textit{multipermutation} is a word $\pi=\pi_1\pi_2\cdots\pi_{m}$ where $m=m_1+m_2+\cdots+m_n$ and $\pi$ contains $i$ exactly $m_i$ times for all $i \in [n]$.
We let $\mathfrak{S}_M$ denote the set of multipermutations on $M$.
As expected, a \emph{descent} in a multipermutation $\pi$ is an index $j$ such that $\pi_j > \pi_{j+1}$.
As before, we let $\mathrm{des}(\pi)$ be the number of descents of $\pi$. Then
\begin{align}
\label{eq:eulerianDef}
    A_{M}(t) = \sum_{\pi \in \mathfrak{S}_M} t^{\mathrm{des}(\pi)}
\end{align}
is defined to be the \emph{multiset Eulerian polynomial}, where the sum is over the set of all multipermutations $\mathfrak{S}_M$.
The coefficients of $A_M(t)$ are known as the \emph{Simon Newcomb numbers}, counting the number of descents in permutations of multisets.
Due to MacMahon~\cite[p.~211]{MacMahon}, $A_M(t)$ occurs as the numerator of the following generating function
\begin{align}
    \label{eq:genmac}
    \frac{A_M(t)}{(1-t)^{m +1}} = \sum_{\ell=1}^{\infty}\prod_{i=1}^{n} \binom{m_i+\ell}{\ell}t^\ell.
\end{align}
Equation~\eqref{eq:genmac} allows for the explicit computation of the coefficients as
\begin{align}
\label{eq:eulercoeff}
[t^k] A_M(t)=\left|\left\{\pi\in \mathfrak{S}_M : \mathrm{des}(\pi)=k\right\}\right| = \sum_{\ell=0}^k(-1)^{\ell}\binom{m +1}{\ell}\prod_{i=1}^n\binom{m_i + k - \ell}{k - \ell},
\end{align}
where $[t^k] A_M(t)$ denotes the coefficient of $t^k$ in $A_M(t)$. Using \eqref{eq:eulercoeff}, we enumerate descents in parking functions, a superset of permutations, and generalize the results of 
Schumacher~\cite{Schumacher}.

Consider the set of non-decreasing parking functions $\NDPF_n$.
By its construction, each $\alpha \in \NDPF_n$ has no descents.
However, each $\alpha$ also gives a multipermutation: consider a multiset $\left\{1^{m_1},\ldots,n^{m_n}\right\}$, where, for all $i\in[n]$, the value $m_i$ is exactly the number of cars preferring spot $i$ in the parking function $\alpha$.
Now consider the generating function
\begin{align}
    \label{eq:genFcNoSol}
    \sum_{n =1}^{\infty}\sum_{j=1}^{n-1}d(n,j)x^ny^j,
\end{align}
where $d(n,j)$ is the number of parking functions in $\PF_n$ that have $j$ descents.
We can now prove Theorem \ref{thm:eulerian} and use it to compute~\eqref{eq:genFcNoSol}. We restate the theorem more succinctly below.

\begin{theorem}\label{thm:eulerian}
Let $n \in \NN$, then the generating function encoding the number of descents throughout all $\PF_n$ is given by
\begin{align}
    \label{eq:genFctSol}
    \sum_{\alpha \in \PF_n} d(n,j)x^ny^j = \sum_{n =1}^{\infty}\sum_{j=1}^{n-1}d(n,j)x^ny^j = \sum_{n=1}^{\infty} \left(\sum_{\beta \in \NDPF_n} A_{\beta}(y) \right) x^n.
\end{align}
\end{theorem}

\begin{proof}
The result follows from Equation~\eqref{eq:eulerianDef}, see also~\cite[Equation~3.7]{dillon} for more details.
\end{proof}
\setcounter{section}{2}
To better understand Theorem~\ref{thm:eulerian}, we present the case where $n=3$.
\begin{example}
\label{ex:multieuler}
    Let $n=3$. We have $\NDPF_3=\left\{(1,1,1),(1,1,2),(1,1,3),(1,2,2),(1,2,3)\right\}$.
    We can partition the set $\PF_3$ into $\Sym_3$-orbits of the elements of $\NDPF_3$, where $\Sym_3$ acts by permuting preferences:
    \[
        \PF_3= \Sym_3 (1,1,1) \bigcupdot \Sym_3 (1,1,2) \bigcupdot \Sym_3 (1,1,3) \bigcupdot \Sym_3 (1,2,2) \bigcupdot \Sym_3 (1,2,3).
    \]
    We can compute $A_p(y)$ for each $p\in\NDPF_3$ using~\eqref{eq:eulercoeff}.
    For example, we find
    \[
        \Sym_3 (1,1,2) = \left\{(1,1,2),(1,2,1),(2,1,1)\right\},
    \]
    and therefore
    $
        A_{(1,1,2)}(y) = 2y+1
    $.
    Summing over all elements of $\NDPF_3$, we retrieve
    \[
        \sum_{p \in \NDPF_3} A_{p}(y) =y^2+10y+5.
    \]
    Similarly, we can recover results from Schumacher~\cite[Figure~2]{Schumacher} (see Figure~\ref{recreate table}), who provides counts for the numbers of parking functions with $i$ ties and $j$ descents.
    Setting $n=6$, we find
    \[
        \mathcal{A}_6 = \sum_{p \in \NDPF_6} A_{p}(y) = y^5 + 422y^4 + 4770y^3 + 8530y^2 + 2952y + 132.
    \]
    Each coefficient in $\mathcal{A}_6$ agrees with the sum of a diagonal of Schumacher's triangle (from the top left to the bottom right): there is $1$ element in $\PF_6$ with $5$ descents, there are $422$ elements in $\PF_6$ with $4$ descents, and so on.
\end{example}

\begin{remark}
    Recall that $|\NDPF_n|=C_n$, where $C_n$ is the $n$th \emph{Catalan number}. 
    Therefore, partitioning $\PF_n$ into $\Sym_n$--orbits of $\NDPF_n$ means partitioning $\PF_n$ into $C_n$ disjoint subsets.
    Additionally, note that setting $t=1$ in $\sum_{p \in \NDPF_n}A_p(t)$ recovers the known enumeration for parking functions: \[(n+1)^{n-1}=\sum_{p \in \NDPF_n}A_p(1).\]
\end{remark}

\begin{problem} It remains an open problem to give formulas for $T_n(i,j)$ which denotes the number of parking functions of length $n$ with $i$ ties and $j$ descents. 
\end{problem}

\section{Descent Sets of Parking Functions}\label{sec:descent sets}

We now consider parking functions with a specified descent or ascent set. For $\alpha=(a_1,a_2,\ldots,a_n)\in\PF_n$, we let 
$\Des(\alpha)=\{i\in[n-1]: a_{i}>a_{i-1}\}$ {and $\Asc(\alpha) = \{i \in [n-1]: a_{i} < a_{i-1}\}$}. 
For $I\subseteq[n-1]$, define 
\begin{align}D(I;n)&=\{\alpha\in\PF_n:\Des(\alpha)=I\},\mbox{ and}\\{A(I;n)}&={\{\alpha \in \PF_n:\Asc(\alpha) = I\}},\end{align}
and denote {their cardinalities} by $d(I;n)=|D(I;n)|$ {and $a(I;n) = |A(I;n)|$.}

Recall that $D(\emptyset;n)=\NDPF_n$, and it is known that $|\NDPF_n|=C_n=\frac{1}{n+1}\binom{2n}{n}$, the $n$th Catalan number \cite[\href{https://oeis.org/A000108}{A000108}]{OEIS}. Hence, $d(\emptyset;n)=C_n$.
For $1\leq n\leq 4$, Table \ref{tab:descent sets} provides data on the number of parking functions with a given descent set.

\begin{table}
\centering
\begin{tabular}{|c|c|c|p{4in}|}
    \hline
         $n$&$I\subseteq[n-1]$&$d(I,n)$&$D(I,n)$     
 \\\hline\hline
 1&$\emptyset$&$C_1=1$&1\\\hline
 \multirow{2}*{2}&$\emptyset$&$C_2=2$&11, 12\\\cline{2-4}
 &$\{1\}$&1&21\\\hline
         \multirow{4}*{3}&$\emptyset$ &$C_3=5$&
         111,
112,
113,
122,
123\\\cline{2-4}
&$\{1\}$&5&
211,
311,
212,
213,
312
\\\cline{2-4}
&$\{2\}$&5&
131,
121,
221,
132,
231
\\\cline{2-4}
&$\{1,2\}$&1&321\\\hline
\multirow{17}*{4}&$\emptyset$&\multirow{2}*{$C_4=14$}&
1111, 
1112,
1113,
1114,
1122,
1123,
1124,
1133,
1134,
1222,
1223,
1224,
1233,
1234\\\cline{2-4}
&\multirow{3}*{$\{1\}$}&\multirow{3}*{21}&
2111,
3111,
4111,
2112,
2113,
3112,
2114,
4112,
3113,
3114,
4113,
2122,
2123,
3122,
4122,
2133,
3123,
2134,
3124,
4123,
2124
\\\cline{2-4}
&\multirow{4}*{$\{2\}$}&\multirow{4}*{31}&
1211,
1311,
1411,
1212,
2211,
1213,
1312,
2311,
1214,
1412,
2411,
1313,
3311,
1314,
1413,
3411,
2212,
1322,
2213,
2312,
1422,
2214,
2412,
1323,
2313,
3312,
1324,
1423,
2314,
2413,
3412
\\\cline{2-4}
&\multirow{3}*{$\{3\}$}&\multirow{3}*{21}&
1121,
1131,
1141,
1221,
1132,
1231,
1142,
1241,
1331,
1143,
1341,
2221,
1232,
2231,
1242,
2241,
1332,
2331,
1243,
1342,
2341
\\\cline{2-4}
&$\{1,2\}$&9&
3211,
4211,
4311,
3212,
4212,
3213,
3214,
4213,
4312
\\\cline{2-4}
&$\{1,3\}$&19&
2121,
2131,
3121,
4121,
3131,
2141,
3141,
4131,
2132,
3221,
2142,
3121,
3231,
4221,
2143,
3142,
3241,
4132,
4231
\\\cline{2-4}
&$\{2,3\}$&9&
1321,
1421,
1431,
2421,
3321,
1432,
2431,
2321,
3421
\\\cline{2-4}
&$\{1,2,3\}$&1&4321
\\\hline
    \end{tabular}
    \caption{Parking functions of length $1\leq n\leq 4$ with a given descents set.}
    \label{tab:descent sets}
    \end{table}

Next, we prove some preliminary results which will help us establish that the number of parking functions with descents at the indices in $I\subseteq[n-1]$ is the same as the number of parking functions with descents at indices in $J=\{n-i:i\in I\}$.

\begin{proposition}\label{prop:flipflop}
    Let $I \subseteq [n-1]$ and $J = \{n-i:i \in I\}$. Then $d(I;n) = a(J;n)$
\end{proposition}
\begin{proof}
    Given $\alpha = (\alpha_1,\alpha_2,\ldots, \alpha_n)$ with $\Des(\alpha) = I$, then $(\alpha_n,\alpha_{n-1},\ldots,\alpha_1)$ has $\Asc(\alpha) = J$. This defines a bijection between the sets $D(I;n)$ and $A(J;n)$. Thus, $d(I;n)=a(J;n)$, as claimed.
\end{proof}

Next we define a function used in our subsequent results.
To simplify notation, we denote tuples in one-line notation and refer to them as words, i.e., the tuple $(w_1, w_2, \ldots, w_n)$ will be denoted by the word $w_1~w_2~\cdots ~w_n$. 

\begin{definition}\label{def:nu}
    Let $\Sym(\y)$ denote the orbit of $\y \in \NN^n$ under the action of the symmetric group permuting indices. 
    Given $\x\in\Sym(\y)$, let $\argmin(\x)$ denote the set of indices of minimal elements of $\x$ and $\argmax(\x)$ denote the set of indices of maximal elements of $\x$.
    
    Define the map $\nu: \Sym(\y) \to \Sym(\y)$ recursively as follows.
    For $\x\in\Sym(\y)$: 

    \begin{enumerate}
    \item[Case 0:]\label{case0} If the length of $\x$ is 0 or 1, then $\nu(\x)=\x$. 
        \item[Case 1:] \label{def:max left} 
        If the length of $\x$ is larger than 1 and $a \geq b$ for all $a \in \argmin(\x)$ and $b \in \argmax(\x)$,
        then let $i = \max(\argmin(\x))$ and $j = \min(\argmax(\x))$. 
        Then 
        \[\nu(\x) = \nu(x_1~x_2~\cdots~x_{j-1}) \,~ x_i \,~ \nu(x_{j+1}~\cdots~  x_{i-1}) \, ~x_j \,~ \nu(x_{i+1} ~\cdots~ x_n).\]

        \item[Case 2:]\label{def: max right} 
        If the length of $\x$ is larger than 1 and there is some $a < b$ with $a \in \argmin(\x)$ and $b \in \argmax(\x)$,  then
        let $i$ be the greatest element of $\argmin(\x)$ that is smaller than some element of $\argmax(\x)$ and let $j$ be the smallest element of $\argmax(\x)$ greater than $i$. 
        Then  
        \[\nu(\x) = \nu(x_1~x_2~\cdots ~x_{i-1}) \, ~x_j~ \, \nu(x_{i+1}~\cdots~ x_{j-1}) \,~ x_i \,~ \nu(x_{j+1} ~\cdots~ x_n).\]
    \end{enumerate}
\end{definition}

The following remark is helpful and uses wording we implement in subsequent proofs.
\begin{remark}\label{rem:useful}
In Definition \ref{def:nu},
    Case 1 only considers words $\x$ whose minimal value(s) all appear to the right of all of its maximal value(s).
    Thus, when applying $\nu$ to $\x$ results in a word $\nu(\x)$ that has every maximal value(s) to the right of every minimal value(s).
    Whenever, $\x$ has a minimal value to the left of a maximal value, then Case 2 will be applied.
    Hence, applying $\nu$ to $\x$ results in $\nu(\x)$ having a minimal value to the left of a maximal value. 
\end{remark}

We now state our main result.

\begin{proposition}\label{prop:NN-AscToDes}
    Fix $\y \in \NN^n$ and let $S(\y; I)$ denote the set of elements of $\Sym(\y)$ with ascent set $I$ and let $R(\y; I)$ denote the set of elements of $\mathfrak{S}(\y)$ with descent set $I$. Then the restriction of the map $\nu$ given by  $\nu: S(\y; I) \to R(\y; I)$ is a bijection.
\end{proposition}

Definition \ref{def:nu} gives the technical definition for the map $\nu$, which we use to prove Proposition~\ref{prop:NN-AscToDes}.
Before giving the technical proof, we illustrate the map $\nu$ in the following example.

\begin{example}\label{ex:nu}
Let $\y=(1,1,1,1,2,3,3,4,4)\in \NN^9$ and
consider $\x = (1,4,1,3,4,3,1,2,1) \in \Sym(\y) $. 
We write $\x$ in one-line notation as:
    $\x = 1~4~1~3~4~3~1~2~1$. 
    Note $\Asc(\x) = \{1,3,4,7\}$,  
    $\argmin(\x)=\{1,3,7,9\}$, and $\argmax(\x)=\{2,5\}$. 
    We use Definition \ref{def:nu} to give $\nu(\x)$.
    \begin{itemize}
        \item Since the length of $\x$ is larger than 1 and there exists $a\in\argmin(\x)$ and $b\in\argmax(x)$ with $a<b$, we are in Case~2.  
        Then $i=3$ and $j=5$. 
        Hence,
        \begin{align}
        \nu(\x) &= \nu(1~4)~\underbrace{4}_{~x_5}~\nu(3)~~\underbrace{1}_{x_3}~~\nu(3~1~2~1).\label{exofnu}
        \end{align}
        We now consider $\nu(1~4)$, $\nu(3)$ and $\nu(3~1~2~1)$ independently.
        \item For the subword $1~4$, we are again in Case 2.
        So,  $\nu(1~4)=4~1$. 
        \item Since $3$ has length $1$, by Case 0,
        $\nu(3)=3$.
        \item For the subword $3~1~2~1$, $\argmin(3~1~2~1)=\{2,4\}$ and $\argmax(3~1~2~1)=\{1\}$. 
        Hence by Case 1, 
        $\nu(3~1~2~1)=1~\nu(1~2)~3$. Then by Case 2, 
        $\nu(1~2)=2~1$. 
    \end{itemize}
    Substituting these findings into Equation~\eqref{exofnu} yields
    \[\nu(1~4~1~3~4~3~1~2~1)=4~1~4~3~1~1~2~1~3
    .\]
    Finally, $\Des(\nu(\x)) = \{1,3,4,7\} = \Asc(\x)$, as desired. 
\end{example}

\begin{lemma}\label{lem:AsctoDes}
    Let $\y\in\NN^n$ and $\x\in\Sym(\y)$. Then 
    $\Asc(\x) = \Des(\nu(\x))$ for $\nu$ as defined in Definition~\ref{def:nu}.
\end{lemma}

\begin{proof}
    We induct on  $n$ which henceforth denotes the length of $\x$ (which equals the length of $\y$). 
    For the base case, if $n=0$ or $n=1$, then by Case 0 
    of Definition~\ref{def:nu}, $\nu(\x)=\x$ and 
    $\Asc(\x) = \emptyset = \Des(\nu(\x))$.

Fix $n>1$, $\y\in\NN^n$, and $\x\in\Sym(\y)$.
 Assume, for strong induction, that for all $n' < n$, $\y' \in \mathbb{N}^{n'}$, and $\x' \in \Sym(\y')$, we have $\Asc(\x') = \Des(\nu(\x'))$. 
 We now consider $\nu(\x)=\nu(x_1~x_2~\cdots~x_n)$ and proceed via a case-by-case analysis given by the cases in Definition \ref{def:nu}.
    \begin{enumerate}
    \item[Case (1)]: Suppose $a \geq b$ for all $a \in \argmin(\x)$ and $b \in \argmax(\x)$.        
    Let $M=\max(\x)$ and let $m=\min(\x)$. 
    Then
        \begin{align*}
            \x &= x_1~ x_{2}~\cdots~x_{j-1}~ \underbrace{M}_{j} ~ x_{j+1}~\cdots~x_{i-1}~ \underbrace{m}_{i} ~x_{i+1}~ \cdots ~x_{n}
        \end{align*}
        where $j = \min(\argmax(\x))$ and $i = \max(\argmin(\x))$.
        From Case 1 in Definition \ref{def:nu},
        \begin{align*}
            \nu(\x) &= \nu(x_1~x_2~\cdots~x_{j-1}) \,~ \underbrace{m}_{j} \,~ \nu(x_{j+1}~\cdots~  x_{i-1}) \, ~\underbrace{M}_{i} \,~ \nu(x_{i+1} ~\cdots~ x_n).
        \end{align*}
        From the inductive hypothesis, for each $\x'\in\{x_1
        ~x_2~\cdots~x_{j-1},~x_{j+1}~x_{j+2}~\cdots~ x_{i-1},~ x_{i+1}~x_{i+2}~\cdots~ x_{n}\}$ we know $\Asc(\x') = \Des(\nu(\x'))$.
        Thus, we only need to examine $\nu(\x)$ at indices $j-1$, $j$, $i-1$, and $i$.
        We proceed with a case-by-case analysis: 
        \begin{enumerate}
            \item       
            Consider the value at index $j-1$: 
            
            If $j = 1$, then $j-1=0$, and hence there is no ascent, descent, or tie in $\x$ at index 0. 
            Thus, the proposition is vacuously true.
            
            If $j>1$, then $\x$ has an ascent at position $j-1$
            because $M=\max(\x)$ and $j=\min(\argmax(\x))$.
            Since $\argmin(\x) \subseteq \{j, j+1, \ldots, i-1, i\}$, there is no minimal element at index $j-1$. 
            Because $\nu$ places $m$ in position $j$, $\nu(\x)$ has a descent at position $j-1$.
            
            \item Consider the values at indices $j$ and $i-1$: 
            
            Note that $\x$ must have a non-ascent at position $j$ because $M$ is a maximal element causing either a tie or descent at $j$. Likewise, there is a non-ascent in position $i$ because $m$ is a minimal element causing either a tie or descent at position $i-1$.
 
            Then swapping $m$ and $M$ in $\nu(x)$, implies the following: First, since $m$ is a minimal element, placing it at position $j$ ensures that at this index we either an ascent or tie. 
            This means $\nu(x)$ has a non-descent at index $j$. Second, since $M$ is a maximal element, placing it at position $i$ ensures that the index prior is less than or equal to $M$. Implying that $\nu(\x)$ has a non-descent at index $i-1$.
  
            \item Consider the value at index $i$: 
            
            If $i = n$, then there is no ascent, descent, or tie at position $i$ in $\x$, and the proposition is vacuously true. 
             
            If $i < n$, the reasoning is symmetric to the case for $j-1$; $\x$ has an ascent at position $i$ and $\nu(\x)$ has a descent at the same position.
        \end{enumerate}

    \item[Case (2)]: Suppose there is some $a < b$ with $a \in \argmin(\x)$ and $b \in \argmax(\x)$. 
        Let $i$ be the greatest element of $\argmin(\x)$ that is smaller than some element of $\argmax(\x)$ and let $j$ be the smallest element of $\argmax(\x)$ greater than $i$. 
        Furthermore, let $M=\max(\x)$ and let $m=\min(\x)$. 
        Then
        \begin{align*}
            \x &= x_1~ x_{2}~\cdots~x_{i-1}~ \underbrace{m}_{i} ~ x_{i+1}~\cdots~x_{j-1}~ \underbrace{M}_{j} ~x_{j+1}~ \cdots ~x_{n}
        \end{align*}
        and from Case 1 in Definition \ref{def:nu},
        \begin{align*}
            \nu(\x) &= \nu(x_1~x_2~\cdots~x_{i-1}) \,~ \underbrace{M}_{i} \,~ \nu(x_{i+1}~\cdots~  x_{j-1}) \, ~\underbrace{m}_{j} \,~ \nu(x_{j+1} ~\cdots~ x_n).
        \end{align*}
        We again only need to examine positions $i-1$, $i$, $j-1$, and $j$.
        We proceed with a case-by-case analysis: 
        \begin{enumerate}
            \item Consider the value at index $i - 1$: 
            
            If $i = 1$, then $i-1=0$, and hence there is no ascent, descent, or tie in $\x$ at index 0. 
            Thus, the proposition is vacuously true. 
            
            If $i>1$, 
            then in $\x$ there is a non-ascent at position $i-1$.
            Then swapping $m$ and $M$, ensures that there is a non-descent at position $i-1$ in $\nu(\x)$.
 
            \item Consider the values at positions $i$ and $j-1$: 
            
            There are no minimal or maximal elements in $x_{i+1} ~\cdots ~x_{j-1}$, but $x_i = m$ and $x_j = M$. 
            Therefore, $\x$ has ascents in positions $i$ and $j-1$.
            
            Then for the same reasons, because $\nu$ places $m$ in position $j$ and $M$ in position $i$, $\nu(\x)$ has descents at positions $i$ and $j-1$.

            \item Consider the value at index $j$: 
            
            If $j = n$, then there is no ascent, descent, or tie 
            at position $j$; and the proposition is vacuously true.

            If $j < n$, the reasoning is symmetric to the case for $i-1$; $\x$ must have a non-ascent at position $j$ and $\nu(\x)$ must have a non-descent at the same position.
        \end{enumerate}
    \end{enumerate}

    Therefore, $\Asc(\x) = \Des(\nu(\x))$.
\end{proof}

\begin{lemma}\label{lem:nu-inj}
    The map $\nu$ of Definition \ref{def:nu} is injective.
\end{lemma}

\begin{proof}
    We now left-invert $\nu$ by inducting on length. 
    In the base case, when the length is zero or 1, then $\nu$ acts as the identity 
    causing the left-inverse to also be the identity, hence injective. 
    Assume for strong induction, that $\nu$ is left-invertible for all inputs of length less than $n$ (i.e.~of length $1<k<n$). 
    Now consider $\nu(\x)$ of length $n$.
    We will left invert $\nu(\x)$ by uniquely identifying $\x$. 
    Let $m = \min(\nu(\x))$ and $M = \max(\nu(\x))$. 
    By Definition \ref{def:nu}, $\nu$ does not change the content of the input.
    Hence, $m=\min(\nu(\x))=\min(\x)$ and $M = \max(\nu(\x))=\max(\x)$.
    Given that $n>1$, then either: 
    \begin{enumerate}
        \item $\x$ has the property of either Case $1$ of Definition \ref{def:nu} and
    \[\nu(\x) = \nu(x_1~x_2~\cdots~x_{j-1}) \,~ \underbrace{m}_{j} \,~ \nu(x_{j+1}~\cdots~  x_{i-1}) \, ~\underbrace{M}_{i} \,~ \nu(x_{i+1} ~\cdots~ x_n),\] or 
    \item $\x$ is of the from in Case $2$ of Definition \ref{def:nu} and 
    \[\nu(\x) = \nu(x_1~x_2~\cdots ~x_{i-1}) \, ~\underbrace{M}_{i}~ \, \nu(x_{i+1}~\cdots~ x_{j-1}) \,~ \underbrace{m}_{j} \,~ \nu(x_{j+1} ~\cdots~ x_n),\]
    \end{enumerate}
    Thus, $i$ and $j$ are indices determined using Definition \ref{def:nu}.

    Given $\nu(\x) \in \nu(\Sym(\y))$, if the indices $i$ and $j$ can be uniquely identified by the structure of the word $\nu(\x)$, then we can inductively construct $\nu^{-1}$ by swapping the $i$th and $j$th entries of $\nu(\x)$ and then applying $\nu^{-1}$ to each of the subwords of $\nu(\x)$ with smaller length.
    This then would uniquely identify the input $\x$, and establish injectivity. 

    The key to this creating the inverse relies on determining whether $\nu(\x)$ satisfies either:
    \begin{enumerate}
        \item[(a)] all minimal elements are to the left of all maximal elements, or 
        \item[(b)] there exists a minimal element to the right of a maximal element. 
    \end{enumerate}
We consider each of the above cases independently. 

For Case (a) where all minimal elements are to the left of all maximal elements in $\nu(\x)$, the first step in applying $\nu$ to $\x$ can not have arisen from Case 2 in Definition \ref{def:nu}, because Case 2 explicitly places a minimal element to the right of a maximal element. Therefore,
the first step in applying $\nu$ to $\x$ arose from Case 1 in Definition \ref{def:nu}. 
This implies that $\x$ has all maximal elements to the left of all minimal elements and in the subword $x_{j+1}~\cdots~ x_{i-1}$, all instances of $M$ are to the left of all instances of $m$. 
        Therefore, if $x_{j+1} ~\cdots~ x_{i-1}$ has instances of both $M$ and $m$, $\nu$ acts on $x_{j+1} ~\cdots~ x_{i-1}$ as in Case 1 of Definition~\ref{def:nu}.
        Then  
   \[\nu(\x) = \nu(x_1~x_2~\cdots~x_{j-1}) \,~ \underbrace{m}_{j} \,~ \nu(x_{j+1}~\cdots~  x_{i-1}) \, ~\underbrace{M}_{i} \,~ \nu(x_{i+1} ~\cdots~ x_n)\]  
        where index $j$ is the smallest index that contains a minimal element and index $i$ is the largest index that contains a maximal element.
        Therefore, $i=\max(\argmax(\nu(\x)))$ and $j=\min(\argmin(\nu(\x)))$.
Then by induction hypothesis, we have uniquely identified $\x$, which gave rise to $\nu(\x)$.

Now consider Case (b) where there exists a minimal element to the right of a maximal element in $\nu(\x)$.
Recall that $\nu(\x)\in\nu(\Sym(\y))$.
By induction on the minimum of the number of minimal and maximal elements in $\y$,
if $\nu(\x)$ was the output from Case 1 in Definition \ref{def:nu}, then $\nu(\x)$ would have all minimal elements to the left of all maximal elements.
But by assumption $\nu(\x)$ satisfies Case~(b).
Therefore,
the first step in applying $\nu$ to $\x$ arose from Case 2 in Definition \ref{def:nu}. 
Then $i$ is the index in $\x$ containing the rightmost minimal element that has a maximal element to its right in $\x$.
        Thus,
        \[\x = x_1~ x_{2}~\cdots~x_{i-1}~ \underbrace{m}_{i} ~ x_{i+1}~\cdots~x_{j-1}~ \underbrace{M}_{j} ~x_{j+1}~ \cdots~ x_{n}.\]
        If a minimal element and a maximal element both appear in the subword $x_{j+1} ~\cdots ~ x_n$, then we know that all instances of $m$ must appear to the right of all instances of $M$, as otherwise $i$ would not have been maximal.
        Then 
        \[\nu(\x) = \nu(x_1~x_2~\cdots ~x_{i-1}) \, ~\underbrace{M}_{i}~ \, \nu(x_{i+1}~\cdots~ x_{j-1}) \,~ \underbrace{m}_{j} \,~ \nu(x_{j+1} ~\cdots~ x_n),\]
        and $\nu(x_{j+1}~\cdots~x_n)$ will cause all instances of $m$ to appear to the left of all instances of $M$ according to Case 1 of Definition \ref{def:nu}.
        Thus, $i$ is the index of the rightmost maximal element of $\nu(\x)$ with a minimal element to its right, and $j$ is the index of the leftmost minimal element to the right of $i$
        in $\nu(\x)$.
        Then by induction hypothesis, we have uniquely identified $\x$, which gave rise to $\nu(\x)$.

Thus, we have shown that whether $\nu(\x)$ satisfies Case (a) or Case (b), we can uniquely recover $\x$ which gives rise to $\nu(\x)$. This establishes that $\nu$ is an injection. 
\end{proof}

We are now ready to prove Proposition \ref{prop:NN-AscToDes}.

\begin{proof}[Proof of Proposition \ref{prop:NN-AscToDes}]
    By Lemma \ref{lem:nu-inj} and Lemma \ref{lem:AsctoDes}, $\nu$ as constructed in Definition \ref{def:nu} is an injective map $\Sym(\y) \to \Sym(\y)$ for any $\y \in \mathbb{Z}^n$ such that $\Asc(\x) = \Des(\nu(\x))$.
    Then since $|\mathfrak{S}(\y)|$ is finite,  $\nu$ must also be a surjection.
    Therefore, $\nu$ is the desired bijection.
\end{proof}

 \begin{remark}
    As defined in Proposition \ref{prop:NN-AscToDes}, $\nu$ maps the ascent set to the descent set, but does not (and cannot possibly) make additional promises about mapping the descent set to the ascent set. In Example \ref{ex:nu}, $\Asc(\x) = \{1, 3, 4, 7\} = \Des(\nu(\x))$, but $\Des(\x) = \{2, 6, 8\}$ and $\Asc(\nu(\x)) = \{2, 5, 6, 8\}$. The map that takes the descent set to the ascent set is $\nu^{-1}$.
\end{remark}

\begin{corollary}\label{cor:PF-AscToDes}
    There is a bijection $D(I;n) \to A(I;n)$.
\end{corollary}

\begin{proof}
    Apply $\nu$ from Proposition \ref{prop:NN-AscToDes} to any parking function.
\end{proof}

\begin{theorem}\label{thm:I and J} The set of parking functions with descent set $I$ and the set of parking functions with descent set $J=\{n-i : i\in I\}$ are in bijection, and hence $d(I;n)=d(J;n)$. 
\end{theorem}

\begin{proof}
This result follows Proposition~\ref{prop:flipflop} and Corollary~\ref{cor:PF-AscToDes}.
\end{proof}

In what follows, we say the sets $I\subseteq[n-1]$ and $J=\{n-i:i\in I\}$ are \textit{self-dual} if $I=J$. We now give a formula for the number of self-dual sets.

\begin{lemma}\label{lem:count selfdual}
The number of sets $I \subseteq [n-1]$ such that $I = \{n-i:i \in I\}$ is $2^{\lfloor n/2 \rfloor}$.
\end{lemma}

\begin{proof}
If $n=1$, then $I=J=\emptyset$ and there is $2^{0}=1$ such set. Fix $n\geq 1$. 
For $i\in[n-1]$ consider the mapping $i\rightarrow n-i$. 
Notice that the set $\{i,n-i\}$ is self-dual. 
If $n$ is even, then the set $\left\{\frac{n}{2}\right\}$ is the only self-dual set of one element. 
By symmetry, we need only consider $1\leq i \leq \left\lfloor \frac{n}{2} \right\rfloor$. 
This gives $\left\lfloor\frac{n}{2}\right\rfloor$ distinct self-dual subsets from which we can construct larger self-dual sets. 
    To produce such a set, we take any union of our two element (and possibly the one element set for even $n$) self-dual sets to produce a larger self-dual set. 
    Thus, the total number of self-dual sets of $[n]$ is $2^{\left\lfloor\frac{n}{2}\right\rfloor}$.
\end{proof}

Our next objective is to give a recursive formula for the number of parking functions of length $n$ with descents set $I\subseteq[n-1]$. 
To make our approach precise, we begin with some general definitions and notation. 
Given a multiset $X$ of positive integers with size $n$, let $W(X)$ denote the the set of multipermutations of $X$. For $I\subseteq[n-1]$, let 
\begin{align*}
D_X(I;n)&=\{w\in W(X)\,:\,\Des(w)=I \},
\end{align*}
and $d_X(I;n)=|D_X(I;n)|$. We now establish our first result.
\begin{lemma}
Given $\beta \in\NDPF_n$
Let $M(\beta)$ be the multiset of entries of $\beta$. If $I\subseteq[n-1]$, then 
\[\sum_{\beta\in\NDPF_n}d_{M(\beta)}(I;n)=d(I;n).\] 
\end{lemma}
\begin{proof}
    This follows from definitions of $D_X(I;n)$ and $d_X(I;n)$, and the fact that all parking functions can be constructed from permuting the entries in elements of $\NDPF_n$.
\end{proof}

Recall that $d(\emptyset;n)=C_n$, the $n$th Catalan number. 
We now give a recursion for the number of parking functions with a nonempty descent set. 
 
\begin{theorem} \label{thm:recursion}
    Let $I\subseteq[n-1]$ be nonempty, $m=\max(I)$, and $I^-=I\setminus\{m\}$.  Then
\[d(I;n)=\sum_{\beta\in \NDPF_n}\; \left(\sum_{X\in\mathcal{M}(\beta;m)}d_{X}(I^-;m)\right)-d(I^-;n),\]
where, for $\beta\in\NDPF_n$,  $\mathcal{M}(\beta,m)$ denotes  the collection of multisets consisting of $m$ elements of $\beta$.
\end{theorem}

\begin{proof}
Consider the set $P$ of parking functions $\alpha\in\PF_n$ that can be written as a concatenation $\alpha=\alpha'+\alpha''$ and which satisfy
\begin{enumerate}[leftmargin=.35in]
    \item length of $\alpha'$ is $m$ and the length of $\alpha''$ is $n-m$; and
    \item $\Des(\alpha')=I^-$ and $\Des(\alpha'')=\emptyset$.
\end{enumerate}
We now count the elements of $P$ in two ways. 

First, observe that we can write $P$ as the disjoint union of those $\alpha$ where $\alpha'_m>\alpha''_1$ and those where $\alpha'_m\leq \alpha''_1$.
Hence, \begin{align}
    |P|&=d(I^-;n)+d(I;n).\label{eq P1}
\end{align}

On the other hand, the elements of $P$ can be constructed as follows. 
Consider every $\beta\in\NDPF_n$ as a multiset with $n$ elements.
For a fixed multiset $\beta\in\NDPF_n$, let $\mathcal{M}(\beta,m)$ denote the collection of multisets of size $m$ with entries in $\beta$.

Note that for every 
$X\in\mathcal{M}(\beta,m)$, 
arrange the elements so that they have descent set $I^-$.
This can be done in $d_{X}(I^-;m)$ ways.
The remaining $n-m$ values in $\beta\setminus X$ are used to construct $\alpha''$ by placing those remaining values in nondecreasing order, so as to have no descents. This can be done in a unique way. 
Thus, the number of ways we can construct $\alpha\in P$ is given by \begin{align}
|P|&=\sum_{\beta\in \NDPF_n}\; \left(\sum_{X\in\mathcal{M}(\beta,m)}d_{X}(I^-;m)\right).\label{eq P2}
\end{align}
Solving equation~\eqref{eq P1} for $d(I;n)$ and substituting the right hand-side of equation~\eqref{eq P2} into it yields the 
desired result.
\end{proof}

\begin{example}\label{ex:complicated}
 Let $I=\{1,3\}$. Then
by Theorem \ref{thm:recursion},
\[d(I;4)=\sum_{\beta\in\NDPF_4}\left(\sum_{X\in\mathcal{M}(\beta;3)}d_X(I^-;3)\right) - d(I^-;4).\] 
To compute the sum, select subsets of size $3=\max(I)$ for each multiset $\beta\in\NDPF_4$ and arrange the entries so that they have descents at $I^-=\{1\}$. Table~\ref{tab:complicated}  details the computations establishing
 \[\sum_{\beta\in\NDPF_4}\left(\sum_{X\in\mathcal{M}(\beta,3)}d_X(I^-;3)\right)=40.\]
 Now note that $d(I^-;4)=21$, from which we get $d(I;4)=40-21=19$
 as expected. In fact, Table~\ref{tab:descent sets} gives the elements of $D(J;4)$ for all subsets $J\subseteq[3]$; and in particular, we list the elements in $D(I;4)$ confirming that $d(I;4)=19$.
 \begin{table}[htbp]
     \centering
 \begin{tabular}{|c|l|l|}\hline
\multirow{2}{*}{$\beta\in\NDPF_4$} & \multirow{2}{*}{$\mathcal{M}(\beta,3)$} & \multirow{2}{*}{$\displaystyle\sum_{X\in\mathcal{M}(\beta;3)}d_X(I^-;3)$}\\
&&\\\hline\hline
      $\{1,1,1,1\}$& $\{\{1,1,1\}\}$&0\\\hline
      $\{1,1,1,2\}$ & $\{\{1,1,1\},\{1,1,2\}\}$& $0+1=1$\\\hline
      $\{1,1,1,3\}$ & $\{\{1,1,1\}, \{1,1,3\}\}$ & $0+1=1$\\\hline
      $\{1,1,1,4\}$ & $\{\{1,1,1\}, \{1,1,4\}\}$  &$0+1=1$\\\hline
      $\{1,1,2,2\}$ & $\{\{1,1,2\}, \{1,2,2\}\}$ &$1+1=2$\\\hline
      $\{1,1,2,3\}$ & $\{\{1,1,2\}, \{1,1,2\}, \{1,2,3\}\}$&$1+1+2=4$\\\hline
      $\{1,1,2,4\}$ & $\{\{1,1,2\}, \{1,1,4\}, \{1,2,4\}\}$ &$1+1+2=4$\\\hline
      $\{1,1,3,3\}$ & $\{\{1,1,3\}, \{1,3,3\}\}$&$1+1=2$\\\hline
      $\{1,1,3,4\}$ & $\{\{1,1,3\}, \{1,1,4\}, \{1,3,4\}\}$&$1+1+2=4$\\\hline
      $\{1,2,2,2\}$ & $\{\{1,2,2\}, \{2,2,2\}\}$ &$1+0=1$\\\hline
      $\{1,2,2,3\}$ & $\{\{1,2,2\}, \{1,2,3\}, \{2,2,3\}\}$&$1+2+1=4$\\\hline
      $\{1,2,2,4\}$ & $\{\{1,2,2\}, \{1,2,4\}, \{2,2,4\}\}$&$1+2+1=4$\\\hline
        $\{1,2,3,3\}$ & $\{\{1,2,3\},\{1,3,3\},\{2,3,3\}\}$ &$2+1+1=4$\\\hline
      $\{1,2,3,4\}$& $\{\{1,2,3\},\{1,2,4\},\{1,3,4\},\{2,3,4\}\}$&$2+2+2+2=8$\\\hline
 \end{tabular}
 \caption{Computations for Example \ref{ex:complicated}.}
 \label{tab:complicated}
 \end{table}

\end{example}

Next, we enumerate parking functions with descents at the first $k$ indices by bijecting onto the set of standard Young tableaux of shape $((n-k)^2,1^k)$, which are known to be enumerated by $f(n, n-k-1)=\frac{1}{n}\binom{n}{k}\binom{2n-k}{n-k-1}$ \cite[\href{https://oeis.org/A033282}{A033282}]{OEIS}. 

\begin{proposition}\label{prop:biject to SYT}
Let $n\geq 1$ and $0 \leq k \leq n-1$. If $I=[k]\subseteq[n-1]$, then 
\begin{align*}
    d(I;n) = \frac{1}{n}\binom{n}{k}\binom{2n-k}{n-k-1}.
\end{align*}
\end{proposition}

\begin{proof}
To begin, we recall the bijection between Dyck paths of semilength $n$ and Standard Young tableaux which is given as follows: For $1\leq i\leq 2n$, if the $j$th step in a Dyck path is an Up step (resp.~Down step), then put the value $j$ in the first (resp.~second) row of the standard Young tableau. 
For more on bijections between Dyck paths and standard Young tableaux, we point the interested reader to \cite{pathsSYT}.

    Consider the standard Young tableau
    \[T=\ytableausetup
 {mathmode, boxframe=normal, boxsize=2em}
\begin{ytableau}
 a_1 & a_2 & a_3 & \none[\dots]
              2
& \scriptstyle a_{n-k}\\
 b_1 & b_2 & b_3 & \none[\dots]
              2
& \scriptstyle b_{n-k}\\
c_1\\
c_2\\
       \none[\vdots] \\
   c_k\\
  \end{ytableau}\]
whose entries strictly increase along the columns and rows. 
Construct a word $w = w_1w_2~\cdots ~ w_{2n-k}$ corresponding to a Dyck path, where for all $1\leq i\leq 2n-k$.
For each $1\leq i\leq 2n-k$,  we let
\[w_i = \begin{cases}
    U & \text{if $i$ is in the first row,}\\
    D & \text{if $i$ is in the second row,}\\
    X &  \text{if $i$ is in the third row or beyond.}
\end{cases}\]
Observe, that removing all $X$'s from $w$ would make $w$ a Dyck path, which induces a bijection between Dyck paths and Young tableaux of shape $(n-k,n-k)$, as described above. 
Moreover, inserting a $U~D$ into the middle of a Dyck path returns another Dyck path. 
We construct the word $w'$ of length $2n$ by replacing every $X$ in $w$ with an underlined $\underline{U~D}$. 
As Dyck paths are in bijection with nondecreasing parking functions, we let $\alpha$ be the parking function corresponding to $w'$, underlining the digit that corresponds to the underlined $U$ in $w'$. Construct $\alpha'$ with descent set $[k]$ from $\alpha$ by moving the underlined numbers to the front of $\alpha'$ arranged in decreasing order. 

In the reverse direction, given a parking function $\alpha\in \PF_n$, underline the values in the first $k$ indices.  
Let $\beta$ be the nondecreasing rearrangement of $\alpha$, keeping the underlining of those values. 
If there are repeated values in $\beta$ and one is underlined, then place the underlined value as the right most of those repeated values. This yields a unique $w$, where the underline values give rise to the $X$'s needed to construct the standard tableau. 

In this way, this process is invertible in a unique manner, and the sets are in bijection. 
\end{proof}

We illustrate the bijection in the proof of Proposition \ref{prop:biject to SYT} next.

\begin{example}
Consider $n = 7$, $k = 3$, and the standard Young tableau with shape \\
$((7-3)^2,1^3)=(4, 4, 1, 1, 1)$:
\[T = \begin{ytableau}
    1 & 3 & 7 & 8\\
    2 & 5 & 9 & 10\\
    4\\
    6\\
    11
\end{ytableau}\;.\]
The corresponding word is $w=U~D~U~X~D~X~U~U~D~D~X$.
Replacing $X$'s with $\underline{U~D}$ in $w$ yields
\[w'=U~D~U~\underline{U~D}~D~\underline{U~D}~U~U~D~D~\underline{U~D}.\]
The corresponding $\alpha\in\NDPF_n$ is given by 
$\alpha=(1,2,\underline{2},\underline{4},5,5,\underline{7})$.
Move the underlined numbers in $\alpha$ to the front (in decreasing order) to get
$\alpha'=(7,4,2,1,2,5,5)$ which has the desired descent set $[3]$. 

For an example of the reverse direction consider the parking function 
$\beta=(5,3,2,1,1,2,2)$, which has descent set $[3]$. 
Underline the first $3$ entries and write $\beta$ in nondecreasing order
$(1,1,2,2,\underline{2},\underline{3},\underline{5})$. 
Note that the $2$'s are repeated and the right most $2$ is the underlined value.
From this, construct 
$w'=U~U~D~U~U~\underline{U~D}~\underline{U~D}~D~\underline{U~D}~D~D$.  
Replacing $\underline{U~D}$ with $X$ in $w'$ yields
\[w=U~U~D~U~U~X~X~D~X~D~D.\] 
Then $w$ corresponds to the standard Young tableau
\[T = \begin{ytableau}
    1 & 2 & 4 & 5\\
    3 & 8 & 10 & 11\\
    6\\
    7\\
    9
\end{ytableau}\;.\]
\end{example}

Although we have a recursive formula for the number of parking functions of length $n$ with descent set $I\subseteq[n-1]$ (Theorem \ref{thm:recursion}), we pose the following.
\begin{problem} Can one give non-recursive formulas, such as that presented in Proposition \ref{prop:biject to SYT}, for the number of parking functions with other interesting descent sets $I\subseteq[n-1]$? 
\end{problem}

\section{Peaks of Parking Functions} \label{sec:PTPFN}
The Catalan numbers are one of the most well-studied integer sequences, with $214$ different combinatorial explanations in the book \cite{stanley:catalan}, and many more in entry \cite[\href{https://oeis.org/A000108}{A000108}]{OEIS}. 
We now establish that the set of parking functions of length $n$ that have no peaks and no ties, referred to as peakless-tieless parking functions, are a new set of Catalan objects. 

We begin by setting some needed definitions and notations.
\begin{definition}\label{def:tie and value set}
    Let $\alpha=(a_1,a_2,\ldots,a_n)\in[n]^n$. 
    \begin{itemize}[leftmargin=.15in]
        \item  Define  the
\textit{tie set} of $\alpha$ as
\[\Tie(\alpha)=\{j+1\in[n]: a_{j}=a_{j+1}\}\]
and order the elements 
$\Tie(\alpha)=\{t_1<t_2<\cdots<t_j\}$, where $j=|\Tie(\alpha)|$.
\item Define the \textit{value set} of $\alpha$ as 
\[\Val(\alpha)=\{b'_1, b'_2, \ldots, b'_k\},\]
the set of elements of $\alpha$ in increasing order $1 = b'_1 < b'_2 < \dots < b'_k$, where $k$ is the number of elements of $\alpha$.
    \end{itemize}
\end{definition}

\begin{definition}\label{def:function}
Define the function $\varphi:\NDPF_{n}\to \PTPF_{n}$ by
\begin{align}
    \varphi(\alpha)&=(t_j,t_{j-1},\ldots,t_2,t_1,b_1,b_2,\ldots ,b_k)\label{eq:construction},
    \end{align}
    where $\Tie(\alpha)=\{t_j>t_{j-1}>\cdots>t_1>1\}$ and $\Val(\alpha)=\{1=b_1<b_2<\cdots<b_k\}$ are as in Definition \ref{def:tie and value set}.
    \end{definition}
We illustrate these definitions next.
\begin{example} Consider the nondecreasing parking function 
$\alpha=(1,1,2,3,4,4)\in\NDPF_6$.
Then  $\Tie(\alpha)=\{2, 6\}$ and  $\Val(\alpha)=\{1,2,3,4\}$.
Now $\varphi(\alpha)$ returns the a tuple whose first $|\Tie(\alpha)|$ elements are the integers in $\Tie(\alpha)$ arranged in decreasing order, while the remaining values are the elements of $\Val(\alpha)$ listed in increasing order. Namely $\varphi(\alpha)=(6,2,1,2,3,4)$. 
The result is a peakless-tieless parking function, as claimed.
\end{example}

\begin{theorem}\label{thm:bijection}
    The function $\varphi:\NDPF_n\to \PTPF_n$ in Definition \ref{def:function} is a bijection.
\end{theorem}
\begin{proof}
    \textbf{Well-defined.} Let $\alpha\in\NDPF_{n}$ be as given in Definition \ref{def:function}. 
    In this construction, $|\Val(\alpha)|=k$ is the number of values appearing in $\alpha$ and $|\Tie(\alpha)|=j$ is the number of repeated values appearing in $\alpha$. 
    A value in $\alpha$ either repeats or appears exactly once; this implies that $k+j=n$. 
    Moreover, the construction in \eqref{eq:construction} ensures $b_1=1$ since the letter $1$ must appear in $\NDPF_n$.
    Also, the letters before $1$ (if any) decrease down to $1$, and the letters after $1$ (if any) increase.
    Hence, $\varphi(\alpha)$ has no peaks and no ties. 
    Additionally, note that all cars will still be able to park with preferences $\varphi(\alpha)$. 
    To see this, note that $\alpha=(a_1,a_2,\ldots,a_n) \in \NDPF_n$.
    Cars will park in the order they arrive: car $1$ parks in spot $1$, car $2$ in spot $2$, and so on.
    To show that $\varphi(\alpha)$ parks, we define $\gamma = (g_1, \ldots, g_n)$ from $\alpha$ by replacing each $a_i$ where $a_i=a_{i-1}$ with $i$. 
    Because $\alpha$ parks, each unchanged preference will satisfy $g_i \leq i$, and if $g_j$ is a changed preference, $g_j=j \leq j$.
    Thus, we conclude that $\gamma$ parks. 
    Now, since $\gamma$ and $\varphi(\alpha)$ may be permuted into each other, $\varphi(\alpha)$ must also park. 
    Therefore, the nondecreasing rearrangement of $\varphi(\alpha)=(t_j,t_{j-1},\ldots,t_1,b_1,b_2,\ldots,b_k)$ denoted $(a_1',a_2',\ldots,a_n')$ will satisfy $a_i' \leq i$ for all $i$. 
    Thus, $\varphi(\alpha)\in\PTPF_n$.
    
\noindent    \textbf{Injective.} Let $\alpha,~ \beta \in \NDPF_n$, and let $\varphi(\alpha)=\varphi(\beta)$. 
    This implies that $\Tie(\alpha)=\Tie(\beta)$ and $\Val(\alpha) = \Val(\beta).$
    Notice that nondecreasing parking functions $x \in \NDPF_n$ are completely determined by their tie sets, $\Tie(x)$, and their value sets, $\Val(x)$.
    When reconstructing 
    $\alpha$ and $\beta$ from the value set and the tie set, the repeated values in $x$ will be the same for both $\varphi(\alpha)$ and $\varphi(\beta)$.
    Thus, $\alpha=\beta$.

\noindent    \textbf{Surjective.} Let $\beta = (t_j, t_{j-1}, \ldots, t_2, t_1, b_1, b_2, \ldots, b_k) \in \PTPF_n$, where 
    $T = \{1<t_1< t_2<\cdots<t_{j-1}<t_{j}\}$
    and $B = \{1 = b_1 < b_2< \cdots < b_k\}$. 
    It is sufficient to find $\alpha \in \NDPF_n$ with $\Tie(\alpha) = T$ and $\Val(\alpha) = B$. 
    Order the elements of $X=[n]\setminus T=\{x_1<x_2<\cdots<x_{n-j}\}$.
    Construct 
    $\alpha=(a_1,a_2,\ldots,a_n)$ as follows.
    \begin{enumerate}[leftmargin=.25in]
        \item In $\alpha$, place the elements of $B$ in order $1=b_1<b_2<\cdots<b_k$ at the indices indexed by $X$ in order $x_1<x_2<\cdots<x_{n-j}=x_k$. Namely, let $a_{x_i}=b_i$ for all $1\leq i\leq k$.
        \item For any $j\in T$, set $a_j=a_{j-1}$.
    \end{enumerate}
Since $\alpha = (a_1, a_2, \dots, a_n)$ is a nondecreasing tuple, it parks if and only if $a_i \leq i$ for all $1 \leq i \leq n$. 
    If $a_i = a_{i-1}$ and $a_{i-1} \leq i - 1$, then $a_i \leq i$. 
    Therefore, it is sufficient to check that $a_{x_i} = b_i \leq x_i$ for all $1 \leq i \leq n-j$.

    Let $P=(p_1, p_2, \ldots, p_n)$ be $\beta$ in nondecreasing order.
    This is a parking function, so $p_i\leq i$.
    \begin{enumerate}[leftmargin=.25in]
        \item If there is no duplicate in $(a_1, \ldots, a_i)$: $t_1 >i$.
        Then in $P$, $t_1$ has index at lest $i+1$, implying $\Val(a_1, \ldots , a_i) \subset \Val(p_1, \ldots p_i).$
        \item If there is a duplicate: without loss of generality, we assume it is $a_{i-1}=a_i$.\\
        Then $\Val(a_1, \ldots, a_{i-1}) \subset \Val(p_1, \ldots, p_{i-1})$.
        Notice then, that since $a_{i-1}=a_i$, we have \[\Val(a_1, \ldots, a_{i})=\Val(a_1, \ldots, a_{i-1}) \subset \Val(p_1, \ldots, p_{i-1}) \subseteq \Val(p_1, \ldots, p_i).\]
    \end{enumerate}
    Because of this, we satisfy the claim.
    If we take $a_i \in \alpha$, we can find some $\ell_i \leq i$ such that $a_i = p_{\ell_i} \leq \ell_i \leq i$.
    Then $a_i \leq i$ and $\alpha \in \NDPF_n$.\qedhere
\end{proof}

The bijection in Theorem~\ref{thm:bijection} immediately implies the following. 
\begin{corollary}\label{cor:peakless tieless catalan}
If $n\geq 1$, then $|\PTPF_{n}|=C_n$, the $n$th Catalan number.    
\end{corollary}

We now provide a connection between peakless-tieless parking functions and the entries of the Catalan triangle \cite[\href{https://oeis.org/A009766}{A009766}]{OEIS}, which are defined by
\begin{align}
C_{n+1,i} &= \sum_{j=0}^{i} C_{n,j}.\label{eq:Cat triangle}
\end{align}
Table \ref{tab:cat triangle} provides some of the initial values in the Catalan triangle. 
The sum  along the $n$th row of the Catalan triangle is given by $C_n$.

\begin{table}[htbp]
\centering
\begin{tabular}{|c|c|c|c|c|c|c|c|}\hline
$n\setminus k $&1&2&3&4&5&6&7\\\hline
1&1	&		&		&		&		&		&			\\\hline
2&1	&	1	&		&		&		&		&			\\\hline
3&1	&	2	&	2	&		&		&		&			\\\hline
4&1	&	3	&	5	&	5	&		&		&			\\\hline
5&1	&	4	&	9	&	14	&	14	&		&			\\\hline
6&1	&	5	&	14	&	28	&	42	&	42	&			\\\hline
7&1	&	6	&	20	&	48	&	90	&	132	&	132		\\\hline
\end{tabular}
\caption{The entries of the Catalan triangle for $1\leq n,i\leq 7$.}\label{tab:cat triangle}
\end{table}

\begin{corollary}\label{cor:cat triangle}
    If $\PTPF_{n}(i)=\{\alpha=(a_1,a_2,\ldots,a_n)\in\PTPF_{n}: a_n=i\}$, then 
    $|\PTPF_{n}(i)|=C_{n,i}$.
\end{corollary}
\begin{proof}
We proceed by induction on $n$ and $i\in[n]$.
Note that when $n=1$, the set $|\PTPF_1(1)|=|\{(1)\}|=1=C_{1,1}$. When $n=2$, the set $|\PTPF_2(1)|=|\{(2,1)\}|=1=C_{2,1}$ and $|\PTPF_2(2)|=|\{(1,2)\}|=1=C_{2,2}$. 

Assume for induction that for any $n\leq k$ and $i\leq n$, $|\PTPF_{n}(i)|=C_{n,i}$.
Let $\beta=(b_1,b_2,\ldots,b_{n-1},i)\in\PTPF_{n}(i)$. 
Prepending $n+1$ to $\beta$ yields $(n+1,b_1,b_2,\ldots,b_{n-1},i)\in\PTPF_{n+1}(i)$, and by induction we have constructed $C_{n,i}$ many elements of $\PTPF_{n+1}(i)$. 
For any element $\gamma=(b_1,b_2,\ldots,b_{n-1},j)\in \PTPF_{n}(j)$ with $1\leq j\leq i-1$, appending $i$ to $\gamma$ yields 
$(b_1,b_2,\ldots,b_{n-1},j,i)\in\PTPF_{n+1}(i)$,
and we construct $\sum_{j=1}^{i-1}C_{n,j}$ elements.
These constructions yield distinct elements of $\PTPF_{n+1}(i)$ since the construction creates tuples beginning with $n+1$, while the second begins with $b_1\in[n]$.
As these are the only possible values with which the tuple may begin, applying the induction hypothesis yields
\[|\PTPF_{n+1}(i)|=C_{n,i}+\sum_{j=1}^{i-1}C_{n,j}=\sum_{j=1}^{i}C_{n,j}=C_{n+1,i}\]
where the last equality holds by \eqref{eq:Cat triangle}.
\end{proof}

\subsection{Valleys of Parking Functions}
In permutations, the map $i\to n-i+1$ gives a bijection on permutations, which sends peaks to valleys. 
This establishes that the number of permutations with $k$ peaks is the same as the number of permutations with $k$ valleys. 
However, this map is not well-defined for parking functions. 
For example, $(1,1,1)$ goes to $(3,3,3)$ which is not even a parking function. 
Moreover, the number of parking functions with $k$ peaks is not the same as the number with $k$ valleys. 
For example when $n=4$, there are five peakless-tieless parking functions: \[\PTPF_n=\{(1,2,3),(3,2,1),(2,1,2),(3,1,2),(2,1,3)\}.\] 
Whereas, there are six parking functions with no ties and no valleys: \[(1,2,3),(3,2,1),(1,2,1),(1,3,1),(1,3,2),(2,3,1).\]

We now consider the set of valleyless-tieless parking functions; parking functions for which there is no index $1\leq i\leq n-1$ such that $a_i=a_{i+1}$, nor index $2\leq i\leq n-1$ such that $a_{i-1}>a_i<a_{i+1}$. We let $\VTPF_n$ denote the set of valleyless-tieless parking functions of length $n$. Our main result is as follows.
 
\begin{theorem}\label{thm:main2}
Valleyless-tieless parking functions of length $n$ are enumerated by $|\VTPF_n|=F_{n+2}$, where
$F_n$ is the $n$th Fine number (as defined in \cite{FineNums}) with the first ten values of the sequence being  $(1\leq n\leq 10):$
\[1, 2, 6, 18, 57, 186, 622, 2120, 7338, 25724.\]
\end{theorem}

To prove Theorem \ref{thm:main2}, in Subsection \ref{subsec1},
we establish a bijection between the set of valleyless-tieless parking functions and certain types of peakless-tieless parking functions.  In Subsection \ref{subsec2}, we show that the special set of peakless-tieless parking functions are counted by the Fine numbers. This result will then imply Theorem \ref{thm:main2}.

\subsection{Bijective Maps}\label{subsec1}
We let $\VTPF_{n+1}(a_1=1,a_{n+1}>1)$ be the set of valleyless-tieless parking functions of length $n+1$ in which the first value is one and the last is larger than one. 
Similarly, $\PTPF_{n+1}$
is the set of parking functions of length $n+1$ for which there is no index $i\in[n]$ such that $a_i=a_{i+1}$ or $a_{i-1}<a_i>a_{i+1}$. 
We refer to $\PTPF_{n+1}$ as the set of peakless-tieless parking functions of length $n+1$.
We let 
$\PTPF_{n+1}(a_1>a_{n+1})$ be the set of peakless-tieless parking functions of length $n+1$ in which the first value is larger than the last. 

\begin{lemma}\label{lem:bijection longer star with 1}
The function $\psi:\VTPF_n\to \VTPF_{n+1}(a_1=1,a_{n+1}>1)$ defined by 
\[(a_1,a_2,\ldots,a_n)\to(1,a_1+1,a_2+1,\ldots,a_n+1)\]
is a bijection.    
\end{lemma}

\begin{proof}
First we show injectivity. If $\alpha, \beta \in \VTPF_n$ with $\alpha = (a_1,a_2,\ldots,a_n)$, $\beta = (b_1,b_2,\ldots,b_n)$ and $\psi(\alpha) = \psi(\beta)$, then $a_i+1 = b_i+1$ for all $i \in [n]$. Then $a_i = b_i$ for all $i \in [n]$, so $\alpha = \beta$. 

For surjectivity, suppose $\alpha = (1, a_1, \ldots, a_n) \in \VTPF_{n+1}(a_1 = 1, a_{n+1}>1)$. 
We wish to show that $\beta=(a_1-1,a_2-1,\ldots,a_n-1)\in\VTPF_n$. 
Note that for $i \in [n]$, $a_i > 1$ and in general an element of $\VTPF_n$ can only have a 1 in the first and/or last position;
if an instance of 1 has a neighbor both to its left and right and there are no ties, then both of those neighbors are greater than 1, creating a valley and giving rise to a contradiction. 
Recall $\alpha$ satisfies $a_n > 1$, so, $a_i > 1$ for all $i \in [n]$. 
Moreover, $a_i-1\leq n$ for all $i\in[n]$.
Hence, $\beta\in[n]^n$ is a valid parking preference. 

Next, we show that $\beta$ is a parking function. Let $\alpha^\uparrow = (1, x_1, x_2, \ldots, x_n)$ be the nondecreasing rearrangement of $\alpha$. 
Since $\alpha$ is a parking function, we know that $1 < x_i \leq i+1$ for all $i \in [n]$. Then the nondecreasing rearrangement $\beta^\uparrow$ of $\beta$ is $(x_1 - 1, x_2 - 1,\ldots, x_n-1)$, which satisfies 
$0 < x_n - 1 \leq i$, so $\beta$ parks. 
Removing the first element of any parking function, in particular of $\alpha$, will never create a valley or a tie that was not there before.
This ensures that $\beta$ also does not have any ties nor valleys. 
This then establishes that $\beta\in\VTPF_n$ and satisfies $\psi(\beta)=\alpha$.
\end{proof}

\begin{remark}
    Note that in a valleyless-tieless parking function, there is a unique maximal entry. 
    If this were not the case and the maximum value $k$ appeared twice, then either they are consecutive entries in the tuple, creating a tie (a contradiction), or they are nonadjacent entries in the tuple and thus the value(s) between them would either be larger or create a valley, a contradiction. 
\end{remark}

In the following, we will need to reverse parts of a tuple. 
To this end, we define the following.

\begin{definition}
    Let $(x_1,x_2,\ldots,x_n)\in[n]^n$. 
    For any $i\in[n]$ let
    \begin{align*}
        &\flick_i(x)=(x_i,x_{i-1},\ldots,x_1,x_n,x_{n-1},\ldots,x_{i+1}) \\
        \intertext{and}
        &\flick_n(x)=(x_n,x_{n-1},\ldots,x_2,x_1).
    \end{align*}
\end{definition}
For example, if $\textbf{x}=(1,2,1,3,5,6,3,2)$, then $\flick_4(\textbf{x})=(3,1,2,1 ,2,3,6,5)$ and $\flick_6(\textbf{x})=(6,5,3,1,2,1,2,3)$.
\begin{proposition}
\label{prop:g-is-l}
Define the map $\varphi: \VTPF_{n+1}(a_1=1,a_{n+1}>1) \to \PTPF_{n+1}(a_1 > a_{n+1})$ as follows:
 If $v=(v_1,v_2,\ldots,v_{n+1})\in \VTPF_{n+1}(a_1=1,a_{i+1}>1)$ and $i\in [n+1]$ is the unique index containing the maximal entry of $v$, then  \[\varphi(v)=\flick_i(v).\]
The map $\varphi$ is a bijection.   
\end{proposition}
\begin{proof}
    \textbf{Injective.} Let $v=(1,v_1,v_2,\ldots,v_{n}),w=(1,w_1,w_2,\ldots,w_{n})\in\VTPF_{n+1}(a_1=1,a_n>1)$ and assume that $\varphi(v)=\varphi(w)$.
    Further, assume that the maximal entry of $v$ occurs at index $i$, while the maximal entry of $w$ occurs at index $j$. 
    There are two cases to consider. 
    \begin{enumerate}[leftmargin=.25in]
        \item 
    If $i=j$, then 
$\varphi(v)=\varphi(w)$ implies that 
\[\flick_i(v)=(v_{i-1},v_{i-2},\ldots,v_1,1,v_n,v_{n-1},\ldots,v_{i})=(w_{j-1},w_{j-2},\ldots,w_1,1,w_n,w_{n-1},\ldots,w_{j})=\flick_j(w).\]
Thus, $v_k=w_k$ for all $k\in[n]$, which implies that $v=w$, as desired.
\item Without loss of generality, assume that $i<j$. By definition of $i$ and $j$, $v_{i-1}$ is the maximal entry in $v$ and $w_{j-1}$ is the maximal entry in $w$. 
Then the position of 1 in $\varphi(v)$ lies at index $i$, while the position of 1 in $\varphi(w)$ lies in at index $j$. 
As the entry $1$ is unique in both $\varphi(v)$ and $\varphi(w)$ and they occur at distinct indices, since by assumption $i<j$, we have shown that $\varphi(v)\neq \varphi(w)$, a contradiction. 
Thus, $v=w$, as desired.
\end{enumerate}
    \textbf{Surjective.} Note that $\flick_i$ is an involution that, as a rearrangement, maps parking functions to parking functions. Let $\alpha = (a_1, a_2, \ldots, a_{n+1}) \in \PTPF_{n+1}(a_1 > a_{n+1})$ with $i$ the index of the unique instance of $1 \in \alpha$. 
    Then it is sufficient to show that $\beta = (b_1, b_2, \ldots, b_{n+1}) = \flick_i(\alpha)$ is an element of $\VTPF_{n+1}(a_1=1, a_{n+1}>1)$. Recall that since $\alpha\in\PTPF_{n+1}(a_1>a_{n+1})$, we have that \begin{align}\label{eq:ineqs}a_1 > a_2 > \dots > a_{i-1} > a_i < a_{i+1} < \dots < a_{n+1}\end{align} and $a_1 > a_{n+1}$. 
    Then by definition of $\flick_i(\alpha)$ and the inequalities in \eqref{eq:ineqs}, we have that  \[1 = a_i < a_{i-1} < \dots < a_2 < a_1 > a_{n+1} > a_n > \dots > a_{i+1}. \]So $\beta=(b_1,b_2,\ldots,b_{n+1})=\flick_i(\alpha)$ ensures that \begin{align}
    \label{inequalities2}
    1 = b_1 < b_2 < \dots < b_{i-1} < b_i > b_{i+1} > b_{i+2} > \dots > b_{n+1}\end{align}
    with $b_1 = 1$ and $b_{n+1} > b_1 = 1$.
    The inequalities in \eqref{inequalities2} imply that $\beta\in \VTPF_{n+1}(b_1 = 1, b_{n+1} > 1)$.
\end{proof}

\subsection{Valleyless-Tieless Parking Functions and the Fine Numbers}\label{subsec2}
 We now give a set partition which will prove useful to prove Theorem \ref{thm:main2}.
    
    \begin{definition} For $n\geq 0$, we partition the set $\PTPF_{n+1}$ as follows:
        \begin{itemize}
            \item $\mathscr{G}_{n+1}=\{\alpha \in \PTPF_{n+1}:a_1 > a_{n+1}\}$, 
            \item $\mathscr{E}_{n+1}=\{\alpha \in \PTPF_{n+1}:a_1=a_n\}$, and 
            \item $\mathscr{L}_{n+1}=\{\alpha \in \PTPF_{n+1}:a_1 < a_{n+1}\}$.
        \end{itemize}
    For convenience $\mathscr{G}_1=\mathscr{L}_1=\emptyset$ and $\mathscr{E}_1=\{(1)\}$. Let $G_{n+1} = |\mathscr{G}_{n+1}|$, $E_{n+1} = |\mathscr{E}_{n+1}|$, and $L_{n+1} = |\mathscr{L}_{n+1}|$.
    \end{definition} 
    As expected, 
    \begin{align}
        G_{n+1}+E_{n+1}+L_{n+1}&=|\PTPF_{n+1}| = C_{n+1},\label{eq:cat}
    \end{align}
    the $(n+1)$th Catalan number, and where the last equality holds by Corollary \ref{cor:peakless tieless catalan}.

\begin{remark}\label{rem:properties of ptpfs}
Let $\alpha=(a_1,a_2,\ldots,a_n)$ be a peakless-tieless parking function of length $n$. If $a_i=n$ for some index $i$, then $i=1$ or $i=n$, but not both. 
This is because if both $a_1=n$ and $a_n=n$,  then $\alpha$ is not a parking function. If $a_i=n$ for $i\in[2,n-1]$, then either $\alpha$ has a peak at $i$ or there is a tie in position $i-1$ or $i$, contradicting that $\alpha$ is peakless and tieless. 
\end{remark}

    \begin{lemma}\label{lemma:l-is-e}
        If $n \geq 1$, then $E_{n+1} = L_n$ and if $n\geq 0$, then  $G_{n+1} = L_{n+1}$. 
    \end{lemma}
    \begin{proof}
    For the first equality, we establish a bijection between $\mathscr{E}_{n+1}$ and $\mathscr{L}_n$, from which the result follows.

        Let $e = (e_1,e_2,\ldots e_{n+1}) \in \mathscr{E}_{n+1}$, so that $e_1=e_{n+1}$. Notice that $e$ is peakless-tieless, so $e_2 < e_1$ and since $e_1=e_{n+1}$, we have that $e_2<e_{n+1}$. 
        Furthermore, removing $e_1$ from $e$ does not create a peak, nor a tie. 
        By Remark \ref{rem:properties of ptpfs}, we note that the largest possible value in $e$ is $n-1$, and it would occur at the endpoints. 
        This implies that $e_1=e_{n+1}\leq n-1$. 
        Hence, 
        the tuple $(e_2,e_3,\ldots,e_{n+1})$ would allow the cars to park. 
        Thus, $(e_2,e_3,\ldots, e_{n+1}) \in \mathscr{L}_n$.

        In the reverse direction, if $\ell = (\ell_1,\ell_2, \ldots, \ell_n) \in \mathscr{L}_n$, then
        the tuple $\ell'=(\ell_n,\ell_1,\ell_2,\dots, \ell_n)$ has no peaks and no ties as $\ell_n>\ell_1$. 
        Moreover, $\ell_n\leq n$, hence the largest value in $\ell'$ is $n$ and all $n+1$ cars can park.
        Therefore, 
        $\ell'\in \mathscr{E}_{n+1}$.

        If $\alpha=(a_1,a_2,\ldots,a_{n+1})\in\mathscr{G}_{n+1}$, the map $\flick_n(\alpha)=(a_{n+1},a_{n},\ldots,a_1)$ gives a bijection from $\mathscr{G}_{n+1}$ to $\mathscr{L}_{n+1}$. Hence, 
        $G_{n+1} = L_{n+1}$. 
    \end{proof}
    
    \begin{theorem}\label{thm:Fine count}
        If $n\geq 0$, then $|\PTPF_{n+1}(a_1<a_{n+1})|=F_{n+1}$, the $(n+1)$th Fine number.
    \end{theorem}
    \begin{proof}
        Note that 
        $\PTPF_{n+1} = \mathscr{C}_{n+1}\bigcupdot\mathscr{E}_{n+1}\bigcupdot\mathscr{L}_{n+1}$,
        so by Equation \eqref{eq:cat}, we have that \[C_{n+1} = G_{n+1} + E_{n+1} + L_{n+1}.\] 
        By Lemma \ref{lemma:l-is-e}, $E_{n+1} = L_n$ and $G_{n+1}=L_{n+1}$, implying
        \[C_{n+1} = 2L_{n+1} + L_n.\]
      The initial values are $L_1=0$; and since  $\mathscr{L}_2=\{(1,2)\}$, we have $L_2=1$.
      In \cite[p. 8]{FineNums}, Deutsch and Shapiro prove the following identity relating the Fine numbers to the Catalan numbers
    \[C_{n+1} = 2F_{n+2} + F_{n+1}\]
    where $F_1=0$ and $F_2=1$.
Together, this implies that $L_{n+1} = F_{n+2}$.
    \end{proof}

The bijections in Lemma \ref{lem:bijection longer star with 1} and Proposition \ref{prop:g-is-l}, along with the the enumerative result of Theorem \ref{thm:Fine count} together imply Theorem \ref{thm:main2}, establishing that the number of valleyless-tieless parking functions is given by a Fine number.

\subsection{Open Problems}
 There are many open problems remaining when considering the set of parking functions of length $n$ with $j$ peaks and $k$ ties, which we denote by 
 $\PF_n(j,k)$. 
Table \ref{tab:no peaks k ties} provides the cardinality of the set $\PF_n(0,k)$ for small values of $n$ and $k$. Note that the column corresponding to $k=0$ gives the Catalan numbers, which we proved in Corollary \ref{cor:peakless tieless catalan}. 

\begin{table}[htbp]
    \centering
    \begin{tabular}{|c|c|c|c|c|c|c|}\hline
         $n \setminus k $ ties& 0& 1 &2 &3 &4&5  \\\hline
         1& 1&  0& 0& 0& 0&  0\\\hline
         2& 2&  1& 0& 0& 0&  0\\\hline
         3& 5&  6& 1& 0& 0&  0 \\\hline
         4& 14&  32& 12& 1& 0&  0\\\hline
         5& 42&  178& 110& 20& 1&  0\\\hline
         6& 132& 1078& 978& 280& 30&  1\\\hline
    \end{tabular}
    \caption{Number of parking functions of length $n$ with 0 peaks and $k$ ties.}
    \label{tab:no peaks k ties}
\end{table}

We prove the following two results related to the value of $\PF_n(0,k)$ with $k=n-2,n-1$.

\begin{lemma}
 If $n\geq 1$, then $|\PF_n(0,n-1)|=1$.
 \end{lemma}
 \begin{proof}
     Any tuple in $[n]^n$ with $n-1$ ties must have the same value at every entry. The only such parking functions is the all ones tuple. 
 \end{proof}
 \begin{lemma}
 If $n\geq 1$, then $|\PF_n(0,n-2)|=n(n+1)$, the $n$th Oblong number \cite[\href{https://oeis.org/A002378}{A002378}]{OEIS}.
\end{lemma}
\begin{proof}
    Note the tuple can begin with $i$ many ones, followed by $n-i$ many repeated values $k$ where $2\leq k \leq i+1$. So the number of possibilities is given by $\sum_{i=1}^ni=n(n+1)/2$. We can reverse the tuple to account for the remaining such preferences. This yields a total of $n(n+1)$, as claimed.
\end{proof}
A general open problem follows.
\begin{problem}
For $n,j,k\in\mathbb{N}$, give recursive or closed formulas for the value of $|\PF_n(j,k)|$.
\end{problem}

One could also consider parking functions of length $n$ with $j$ peaks. This set is given by 
\[\PF_n^{j}=\bigcup_{k=0}^{n-1-2j}\PF_n(j,k).\] Table~\ref{tab:peak count} gives the cardinality of $|\PF_n^j|$ for $0\leq j\leq n\leq 6$. Note that the first column corresponding to $j=0$ peaks is also given by the row sums in Table \ref{tab:no peaks k ties}. We can now pose another open problem.

\begin{problem}
    Characterize and enumerate the set $\PF_n^k$ for general values of $n$ and $k$.
\end{problem}

\begin{table}[htbp]
    \centering
    \begin{tabular}{|c|c|c|c|c|}\hline
         $n \setminus j $  peaks& 0& 1 &2 &3  \\\hline
         1& 1&  0& 0& 0\\\hline
         2& 3&  0& 0& 0\\\hline
         3& 12&  4& 0& 0 \\\hline
         4& 59&  66& 0& 0\\\hline
         5& 351&  825& 120& 0\\\hline
         6& 2499& 9704& 4604& 0\\\hline
         7& 20823& 115892& 115959& 9470\\\hline
         8& 197565& 145478& 2479110& 651816\\\hline
    \end{tabular}
    \caption{Number of parking functions of length $n$ with $j$ peaks.}
    \label{tab:peak count}
\end{table}

\section{Statistic 
Encoding}\label{sec:statistic encoding}
We now study a collection of parking functions with a prescribed pattern at every index. 

\begin{definition}
Every $\alpha = (a_1, a_2, \ldots, a_n) \in \PF_n$ gives rise to a word $w = w_1w_2\cdots w_{n-1}\in\{\A,\D,\T\}^{n-1}$, where for each $i\in[n-1]$ and we let
\[w_i=\begin{cases}
    \A&\mbox{ if }a_i < a_{i+1}\\
    \D&\mbox{ if }a_i > a_{i+1}\\
   \T&\mbox{ if }a_i = a_{i+1}.\\
\end{cases}\]
We call $w$ the \emph{statistic encoding} of $\alpha$.  
\end{definition}
In defining a statistic encoding, we use the letter $\A$ for ascent, $\D$ for descent, and $\T$ for~tie. 
The vast majority of statistic encodings are non-unique. For example, $\alpha = (1,1,2,3,4)$ and $\beta = (1,1,3,4,5)$ both have $w = \T\A\A\A$ as their statistic encoding. 

In this section, we answer the following questions: 
\begin{enumerate}[leftmargin=.25in]
    \item Does there exist a parking function with statistic encoding $w$ for arbitrary $w$? \label{q:1}
    \item When is a statistic encoding determined by a unique parking function?\label{q:2} 
\end{enumerate}

To begin, we set the following notation. 
Let $W_{n-1}$ denote the set of all statistic encodings of length $n-1$, which implies $W_{n-1}=\{\A,\D,\T\}^{n-1}$. 
Our first result establishes that every word of length $n-1$ in the letters $\A,\D,\T$ arises as a statistic encoding for some parking function. 
Before proving the result, we illustrate it with an example.

\begin{example}
    To construct a parking function for the word $\D\A\T\A$, we begin with a parking function whose statistic encoding is  $\D\A\T$, such as $\alpha=(3,1,2,2)$. To account for the added ascent at the end of the word $\D\A\T\A$, we simply append $5$ to $\alpha$ and obtain $(3,1,2,2,5)\in\PF_5$. Notice that the word $\D\A\T\A$ is not unique as we could have also appended $3$ or $4$.

    If instead we want to construct a parking function with statistic encoding 
    $\A\T\A\D\D$, we begin with the parking function $\beta=(1,2,2,3,1)$ with a corresponding statistic encoding $\A\T\A\D$. 
    To construct a parking function with statistic encoding $\A\T\A\D\D$, we must create a new descent at the end of $\beta$. 
    To do this, we begin by incrementing every entry in $\beta$ by one, resulting in the tuple $(2,3,3,4,2)$. Then append  $1$ to the end of that tuple creating $(2,3,3,4,2,1)$, which is an element of $\PF_6$ and has the desired statistic encoding $\A\T\A\D\D$.
\end{example}

We are now ready to settle Question~\eqref{q:1}.

\begin{theorem}\label{thm:stat ecoding}
If $w\in W_{n-1}$, then there exists $\alpha\in\PF_n$ which has $w$ as its statistic encoding.
\end{theorem}
\begin{proof}
    We proceed by induction on $n$, the length of the parking function. 
    We begin with the base case where $n=2$ and observe that the statistic encodings $\T, \A$, and $\D$ arise from the parking functions $(1,1)$, $(1,2)$, and $(2,1)$, respectively.

    Assume for induction that for any 
    $w\in\{\A,\D,\T\}^{n-1}$ 
    there is parking function 
    $\alpha=(a_1,a_2,\ldots,a_{n})\in\PF_n$ 
    with statistic encoding $w$.

    Now consider a word 
    $w\in W_{n}=\{\A,\D,\T\}^{n}$ such that $w=w'x$ where $w' \in \{\A,\D,\T\}^{n-1}$ and $x \in \{\A, \D, \T\}$.
    By the inductive step, we can find some $P' = (p'_1, p'_2, \ldots, p'_{n}) \in \PF_{n}$
    whose statistic encoding is $w'$.
    To obtain a parking function $\alpha\in\PF_{n+1}$ with statistic encoding $w=w'x$, we append a new preference to $P'\in\PF_{n}$ based on the letter $x$ by the following criteria:
    \begin{enumerate}[leftmargin=.3in]
    \item[C1:]\label{C1} If $x=\T$, we append $p_{n}'$  to $P'$ constructing 
    \[\alpha = (p_1', p_2', \ldots, p_{n}', p_{n}').\]
    Note $\alpha$ is a parking function of length $n+1$ as $P'$ is a parking function of length $n$, and appending the value $p_n'$ ensures that car $n+1$ parks in spot $n+1$. Moreover, this ensures that the parking function ends with a tie.
    \item[C2:]\label{C2} {If $x=\A$}, we append $n+1$ to $P'$ to get 
    \[\alpha = (p'_1, p'_2, \ldots, p'_n,n+1).\] 
    Note $\alpha$ is a parking function of length $n+1$ as $P'$ is a parking function of length $n$, and appending the value $n+1$ ensures that car $n+1$ parks in spot $n+1$. Moreover, this ensures that the parking function ends with an ascent.
    \item[C3:]\label{C3} {If $x=\D$}, we modify $P'$ by incremented each $p_i'$ by one and then append the value  $1$ at the end of $P'$ to get
    \[\alpha =(p'_1 + 1, p'_2 + 1, \ldots, p'_{n} + 1, 1).\]
    Note $P'$ is a parking function of length $n$, and by incrementing its values by one, the cars $1$ through $n$ in $\alpha$ park in spots $2$ through $n+1$. Then car $n+1$ with preference $1$ parks in spot~$1$. Thus, $\alpha$ is a parking function of length $n+1$. Moreover, this ensures that the parking function ends with a descent.\qedhere
    \end{enumerate}
\end{proof}
In the next result, we use the notation $\A^*\T^*$ to describe a word with some nonnegative integer number of $\A$'s followed by nonnegative integer number of $\T$'s. When we specify that $\A^*\T^*\in W_{n-1}$, then the total number of $\A$'s and $\T$'s must be equal to $n-1$.
Whenever we want to specify the full set of such words, we write $\{\A^*\T^*\}$.
Likewise for $\T^*\D^*$ and $\{\T^*\D^*\}$. With this notation in mind we now settle Question~\eqref{q:2}.

\begin{theorem}
    Let $w \in W_{n-1}$. Then there is a unique $\alpha\in\PF_n$ with statistic encoding $w$ if and only if $w \in \{\A^*\T^*\} \cup \{\T^*\D^*\}$.
\end{theorem}
\begin{proof}
($\Rightarrow$) Note that the operations of appending or prepending an ascent, descent, or tie to a parking function by criteria C1-C3 preserves distinctness, so it is sufficient to avoid certain patterns.
 
Observe:   
    \begin{itemize}[leftmargin=.25in]
        \item The parking functions $(1,3,2)$ and $(2,3,1)$ both have encoding $\A\D$.
        \item The parking functions $(2,1,3)$ and $(3,1,2)$ both have encoding $\D\A$.
        \item The parking functions $(1,1,2)$ and $(1,1,3)$ both have encoding $\T\A$.
        \item The parking functions $(2,1,1)$ and $(3,1,1)$ both have encoding $\D\T$.
        \item The parking functions $(1,2,\dots, 2,1)$ and $(1,3,\dots,3,1)$ both have encoding $\A\T^*\D$.
    \end{itemize}
    These cases imply that any statistic encoding containing $\A\D$, $\D\A$, $\T\A$, $\D\T$, or $\A\T^*\D$ as subword, will give multiple parking functions with that statistic encoding. Hence, to determine which statistic encoding uniquely determine a parking function, we must 
consider those statistic encoding $W_k\in \{\A,\D,\T\}^k$ which have the following forms:
    \begin{itemize}[leftmargin=.25in]
        \item If $w\in W_k$ starts with $\D$, it must be the all-$\D$'s. This holds because if $w$ starts with $\D$, then to uniquely determine a parking function, $\D$ cannot be followed by $\A$ or~$\T$.
        \item If $w\in W_k$ starts with $\A$, it must be followed by only $\A$'s and $\T$'s, and any $\T$ must be followed only by $\T$'s.
        \item If $w\in W_k$ starts with $\T$, it must be followed by only $\T$'s and $\D$'s, and any $\D$ must be followed only by $\D$'s.
    \end{itemize}
Then the only statistic encodings which uniquely determine parking functions must satisfy these three conditions.
This is precisely the desired set $\{\A^*\T^*\} \cup \{\T^*\D^*\}$.

($\Leftarrow$)
    Let $w \in \{\A^*\T^*\} \cup \{\T^*\D^*\}$. We consider the case where $w\in \{\A^*\T^*\}$ and $w\in\{\T^*\D^*\}$, separately and show that $\alpha$ with statistic encoding $w$ is unique.
    \begin{itemize}[leftmargin=.15in]
        \item If $w=\A^i\T^{n-1-i}$ for some $i\in[n-1]$, then $a_1=1$. 
    We claim $a_j=j$ for all $2\leq j\leq i$. Assume for contraction that there exists $1<j\leq i$ such that $a_j\neq j$.  Then the number $j$ will not appear in $\alpha$ as its statistic encoding is  $w=\A^i\T^{n-1-i}$. This would imply that no car parks in spot $j$ and hence $\alpha$ is not a parking function, a contradiction. 
    Thus, $a_j=j$ for all $1\leq j\leq i$. 
    Since $w=\A^i\T^{n-1-i}$, we now have that $a_{i+1}=a_{i+2}=\cdots=a_{n}=i$. Thus, if $w=\A^i\T^{n-1-i}$, then $\alpha$ is unique and has the form
    \begin{align}
        \alpha=(1,2,\ldots,i-1,i,i+1,i+1,\ldots,i+1)\in\PF_n.\label{eq:first alpha}
    \end{align}
    \item If $w=\T^{n-1-i} \D^{i}$ for some $i\in[n-1]$, then $a_n=1$.
    We can reverse the parking function $\alpha$ given in Equation \eqref{eq:first alpha} to create \[\alpha^*=(i+1,i+1,\ldots,i+1,i,i-1,\ldots,3,2,1),\] which is a parking function and has $w$ as its statistic encoding. As $\alpha$ was unique, so is $\alpha^*$.\qedhere
    \end{itemize}
\end{proof}

We conclude by posing the following open problem.
\begin{problem}
    Fix a statistic encoding $w\in W_{n-1}$. Characterize and enumerate the subset of parking functions of length $n$ which have $w$ as their statistic encoding.
\end{problem}
\begin{problem}
For which statistic encoding $w\in W_{n-1}$ is the subset of parking functions of length $n$ which have $w$ as their statistic encoding the largest?
\end{problem}

\bibliographystyle{amsplain}
\bibliography{bibliography}
\end{document}